\documentclass[fullpage,11pt]{amsart}
\usepackage{hyperref,pb-diagram,amssymb,epic,eepic,verbatim,graphicx,graphics,epsfig,psfrag}
\hypersetup{colorlinks}

\usepackage{color}

\definecolor{darkred}{rgb}{0.5,0,0}
\definecolor{darkgreen}{rgb}{0,0.5,0}
\definecolor{darkblue}{rgb}{0,0,0.5}

\hypersetup{ colorlinks,
linkcolor=darkblue,
filecolor=darkgreen,
urlcolor=darkred,
citecolor=darkblue }

\newtheorem{theorem}{Theorem}[section]
\newtheorem{corollary}[theorem]{Corollary}

\newtheorem{conjecture}[theorem]{Conjecture}

\newtheorem{proposition}[theorem]{Proposition}
\newtheorem{lemma}[theorem]{Lemma}
\newtheorem{lem}[theorem]{}
\theoremstyle{definition}
\newtheorem{definition}[theorem]{Definition}

\theoremstyle{remark}
\newtheorem{remark}[theorem]{Remark}
\newtheorem{example}[theorem]{Example}

\newcommand{\blem}{\begin{lem} \rm}
\newcommand{\elem}{\end{lem}}
\newcommand\M{\mathcal{M}}
\newcommand\D{\mathcal{D}}

\newcommand{\R}{\mathbb{R}}

\newcommand{\C}{\mathbb{C}}

\newcommand{\Z}{\mathbb{Z}}
\newcommand{\Q}{\mathbb{Q}}

\renewcommand{\P}{\mathbb{P}}
\newcommand{\PP}{\mathcal{P}}
\newcommand{\Par}{\on{Par}}

\newcommand\lie[1]{\mathfrak{#1}}

\newcommand{\g}{\lie{g}}

\renewcommand{\t}{\lie{t}}

\newcommand{\on}{\operatorname}
\newcommand{\ainfty}{{$A_\infty$\ }}

\newcommand{\dual}{\vee}

\newcommand{\Ve}{\on{V}}
\newcommand{\Edge}{\on{E}}

\newcommand{\Ver}{\on{Vert}}

\newcommand{\Hom}{ \on{Hom}}

\renewcommand{\ker}{ \on{ker}}
\newcommand{\coker}{ \on{coker}}

\newcommand{\Vol}{\on{Vol}}

\newcommand\dirac{/\kern-1.2ex\partial} 
\newcommand\qu{/\kern-.7ex/} 
\newcommand\lqu{\backslash \kern-.7ex \backslash} 

\newcommand\dr{r_+ \kern-.7ex - \kern-.7ex r_-}
 
\usepackage{pb-diagram,amssymb,epic,eepic,verbatim,graphicx}
\usepackage{color}
\usepackage{graphicx}
\usepackage{graphics}
\newcommand{\Ga}{{\Gamma}}

\newcommand{\bA}{\mathbb A}

\newcommand{\ccomment}[1]   {{}}
\newcommand{\labell}\label

\renewcommand{\d}{{\mbox{d}}}
\newcommand{\ol}{\overline}

\newcommand\eps{\epsilon}

\newcommand{\f}{\frac}
\newcommand{\lan}{\langle}
\newcommand{\ran}{\rangle}
\newcommand{\hh}{{\f{1}{2}}}

\newcommand{\ti}{\tilde}

\newcommand\mE{\mathcal{E}}

\newcommand\Map{\on{Map}}

\newcommand\ev{\on{ev}}
\newcommand\Eul{\on{Eul}}

\newcommand\ul{\underline}

\newcommand\Ker{\on{Ker}}

\newcommand\bra[1]{ < \kern-.7ex {#1} \kern-.7ex >} 
\newcommand\bdefn{\begin{definition}}
\newcommand\edefn{\end{definition}}
\newcommand\bea{\begin{eqnarray*}}
\newcommand\eea{\end{eqnarray*}}
\newcommand\bcv{\left[ \begin{array}{r} }
\newcommand\ecv{\end{array} \right] }

\newcommand\bma{\left[ \begin{array} }
\newcommand\ema{\end{array} \right]}
\newcommand\ben{\begin{enumerate}}
\newcommand\een{\end{enumerate}}
\newcommand\beq{\begin{equation}}
\newcommand\eeq{\end{equation}}
\newcommand\bex{\begin{example}}
\newcommand\bsj{\left\{ \begin{array}{rrr} }
\newcommand\esj{\end{array} \right\}}

\newcommand\eex{\end{example}}

\renewcommand\M{{\mathcal{M}}}
\newcommand\fr{{\on{fr}}}
\newcommand\nodes{g}

\newcommand\sx{*\kern-.5ex_X}

\def\mathunderaccent#1{\let\theaccent#1\mathpalette\putaccentunder}
\def\putaccentunder#1#2{\oalign{$#1#2$\crcr\hidewidth \vbox
to.2ex{\hbox{$#1\theaccent{}$}\vss}\hidewidth}}

\begin{document}

\title[Morphisms of CohFT algebras and quantum Kirwan]{Morphisms of CohFT
  algebras and \\ quantization of the Kirwan map}

\author{K. L. Nguyen} 

\address{Department of Mathematics,
Stanford University, 
building 380, Stanford, California 94305
{\em E-mail address: khoan@math.stanford.edu}}

\author{C. Woodward}

\address{Department of Mathematics, 
Rutgers University,
110 Frelinghuysen Rd,
Piscataway, NJ 08854.
{\em E-mail address: ctw@math.rutgers.edu}}

\author{F. Ziltener}

\address{KIAS,
Mathematics, room 1430,
Hoegiro 87 (207-43 Cheongryangni-dong),
Dongdaemun-gu,
Seoul 130-722,
Korea.
{\em E-mail address: 
 fabian@kias.re.kr}}

\thanks{K.N. was supported by Undergraduate Research Experience portion
of NSF grant DMS060509.  C.W. was partially supported by NSF
  grants DMS060509, DMS0904358, and the Simons Foundation.}

\begin{abstract}  
We introduce a notion of morphism of CohFT algebras, based on the
analogy with $A_\infty$ morphisms.  We outline the construction of a
``quantization'' of the classical Kirwan morphism to a morphism of
CohFT algebras from the equivariant quantum cohomology of a
$G$-variety to the quantum cohomology of its geometric invariant
theory or symplectic quotient, and an example relating to the orbifold
quantum cohomology of a compact toric orbifold.  Finally we identify
the space of Cartier divisors in the moduli space of scaled marked
curves; these appear in the splitting axiom.
\end{abstract}

\maketitle

 \tableofcontents

\section{Introduction} 

In order to formalize the algebraic structure of Gromov-Witten theory
Kontsevich and Manin introduced a notion of {\em cohomological field
theory} (CohFT), see \cite[Section IV]{man:fro}.  The correlators of
such a theory depend on the choice of cohomological classes on the
moduli space of stable marked curves and satisfy a splitting axiom for
each boundary divisor.  In genus zero the moduli space of stable
marked curves may be viewed as the complexification of Stasheff's {\em
associahedron} from \cite{st:ho}, and the notion of CohFT may be
related to the notion of $A_\infty$-algebra: dualizing one of the
factors gives rise to a collection of multilinear maps that we call a
{\em CohFT algebra}.  The full CohFT is related to the CohFT algebra
in the same way that a Frobenius algebra is related to the underlying
algebra.  Recall that Dubrovin \cite{dub:geom} constructed from any
CohFT a {\em Frobenius manifold}, which is a manifold with a family of
multiplications on its tangent spaces together with some additional
data.

In this paper we introduce a notion of {\em morphism of CohFT
  algebras} which is a ``closed string'' analog of a morphism of
\ainfty-algebras.  The additional data in the structure maps is the
choice of cohomological classes on {\em moduli space of stable scaled marked lines}
introduced in Ziltener \cite{zilt:phd}.  This space was studied in
Ma'u-Woodward \cite{mau:mult} and identified with the complexification
of Stasheff's multiplihedron which appears in the definition of
$A_\infty$ map \cite{st:ho}.  The splitting axiom for a morphism of
CohFT algebras gives a relation on the structure maps for each divisor
relation.  Any morphism of CohFT algebras has the property that the
linearization at any point is an algebra morphism in the usual sense.
This fits in well with the language of Hertling-Manin \cite{her:weak}
of {\em $F$-manifolds}.

The definition of morphism of CohFT algebra is motivated by an attempt
to extend the mirror theorems of Givental and others beyond the case
of semipositive toric quotients in Givental \cite{giv:tmp}, as has
also been discussed by many authors, for example, Iritani
\cite{iri:gmt}.  In the second part of the paper we describe a {\em
  quantum Kirwan} morphism of CohFT algebras from the equivariant
quantum cohomology $QH_G(X)$ of a smooth polarized projective
$G$-variety $X$ to the (possibly orbifold) quantum cohomology $QH(X
\qu G)$ of the symplectic/git quotient $X \qu G$.  The existence of
this morphism depends on results of the last two authors and
Venugopalan on existence of virtual fundamental classes, see
\cite{qkirwan}.  Morphisms of CohFT algebras provide an ``algebraic
home'' for the counts of ``vortex bubbles'' that first appeared in the
study by Gaio-Salamon \cite{ga:gw} of the relationship between gauged
Gromov-Witten invariants of a $G$-variety and the Gromov-Witten
invariants of the quotient $X \qu G$ \cite{qkirwan}.  Applying the
quantum Kirwan morphism to the special case of quotients of vector
spaces by tori, one obtains a Batyrev-style presentation of the
(possibly orbifold) quantum cohomology of a toric Deligne-Mumford stack at a
special point; this reproduces partial results by
Coates-Corti-Lee-Tseng \cite{coates:wps}.  We discuss several
conjectures ({\em quantum Kirwan surjectivity} and {\em quantum
  reduction in stages}) which arise naturally in this context.  In the
last part of the paper, we describe which combinations of boundary
divisors in the moduli space of stable scaled lines are Cartier, that
is, have dual cohomology classes.

We thank Ezra Getzler, Sikimeti Ma'u, Joseph Shao, and Constantin
Teleman for helpful discussion and comments.

\section{Morphisms of CohFT algebras}  

In this section we describe the definition of morphisms of CohFT
algebras.  Let $\ol{M}_n$ denote the Grothendieck-Knudsen moduli space
of isomorphism classes of genus zero $n$-marked stable curves
\cite{kn:proj2}, which is a smooth projective variety of dimension $
\dim(\ol{M}_n) = n - 3 .$ 

\begin{remark} {\rm (Boundary divisors for the Grothendieck-Knudsen space)}  
The boundary of $\ol{M}_{n}$ consists of the following divisors: for
each splitting $ \{ 1 ,\ldots , n \} = I_1 \cup I_2$ with $|I_1|,|I_2|
\ge 2$ a divisor
$$\iota_{I_1 \cup I_2}: D_{I_1 \cup I_2} \to \ol{M}_{n}$$ 
corresponding to the formation of a separating node, splitting the
curve
into irreducible components
with markings $I_1,I_2$.  The divisor $D_{I_1 \cup I_2}$ is isomorphic
to $\ol{M}_{|I_1|+1} \times \ol{M}_{|I_2|+1}$.  Let
$$\delta_{I_1 \cup I_2} \in H^2(\ol{M}_n)$$ 
denote its dual cohomology class.  For any $\beta \in H(\ol{M}_n)$,
let
\begin{equation} \label{kunneth} 
 \iota_{I_1 \cup I_2}^* \beta = \sum_j \beta_{1,j} \otimes
 \beta_{2,j} \end{equation}
denote the K\"unneth decomposition of its restriction to $D_{I_1 \cup
  I_2}$.
\end{remark} 

\begin{definition}  {\rm (CohFT algebras)}  
An {\em (even, genus zero) cohomological field theory algebra} over a
$\Q$-ring $\Lambda$ is a datum $(V,(\mu^n)_{n \ge 2})$ where $V$ is a
$\Lambda$-module and $(\mu^n)_{n \ge 2}$ is a collection of
multilinear {\em composition maps}
$$ \mu^{n}: V^{n} \times H(\ol{M}_{n+1},\Lambda) \to V $$
such that each $\mu^n$ is invariant under the natural
\ccomment{changed wording slightly} action of the symmetric group
$S_n$ and the maps $(\mu^n)_{n \ge 2}$ satisfy a {\em splitting
  axiom}: for each partition $I_1 \cup I_2 = \{ 1,\ldots, n \}$,
$$ \mu^n(v_1,\ldots,v_n;\beta \wedge \delta_{I_1 \cup I_2})
= \sum_j \mu^{|I_2| + 1} ( \mu^{|I_1|}(v_i, i \in I_1; \beta_{1,j} ),
v_i, i \in I_2; \beta_{2,j} ) $$
where $\beta_{1,j}, \beta_{2,j}$ are as in \eqref{kunneth}.
\end{definition}  

\begin{remark}   It would be more natural to use tensor products
in the above formula but the use of symbols $\otimes$ instead of
commas $,$ makes the formulas substantially longer.
\end{remark} 

\begin{remark} {\rm (Filtered CohFT algebras)}  In our applications, $\Lambda$
will be a {\em filtered $\Q$-ring} by which we mean a union of
decreasing rings $\Lambda_a, a \in \R$:
$$ \Lambda= \cup_{a \in \R} \Lambda_a, \quad \Lambda_a \supset
\Lambda_{b} \ \forall a < b, \quad \cap_{a \in \R} \Lambda_a = \{ 0
\} .$$
A filtered CohFT algebra is a CohFT algebra $V$ with a filtration
$(V_a)_{a \in \R}$ compatible with the $\Lambda$-module structure,
that is, such that $\Lambda_{a} V_b \subset V_{a + b}, \forall a,b \in
\R $, and such that each structure map $\mu^n$ maps $V_a^n \times
H(\ol{M}_{0,n})$ to $V_a$, for all $a \in \R$.
\end{remark}

\begin{remark} 
\begin{enumerate} 
\item {\rm (Comparison with \ainfty algebras)} The collection of
  composition maps $(\mu^n)_{n \ge 2}$ (which are termed in Manin
  \cite{man:fro} $\on{Comm}_\infty$-structures) may be viewed as
  ``complex analogs'' of the \ainfty-structure maps of Stasheff, in
  the sense that the relevant moduli spaces have been
  ``complexified''.
\item {\rm (Relations via divisor equivalences)} The various relations
  on the divisors in $\ol{M}_{n}$ give rise to relations on the maps
  $\mu^{n}$. In particular the relation
$ [D_{ \{0, 3 \} \cup \{ 1,2 \}}] = [D_{ \{ 0,1 \} \cup \{ 2, 3\} }] $
in $H^2(\ol{M}_{4}) $ implies that
$\mu^2: V \times V \to V$
is associative. 
\end{enumerate}
\end{remark} 

The notion of {\em morphism} of CohFT algebras is based on the
geometry of the {\em complexified multiplihedron} $\ol{M}_{n,1}(\bA)$
introduced in Ziltener \cite{zilt:phd} and studied further in
Ma'u-Woodward \cite{mau:mult}.

\begin{definition} 
\label{smoothscalings}
{\rm (Scalings on smooth curves)} 
\begin{enumerate} 
\item A {\em non-degenerate scaling} on a smooth genus zero complex
  projective curve $C$ with {\em root marking} $z_0$ is a meromorphic
  one-form $\lambda: C \to T^\dual C$ with the property that $\lambda$
  has a single pole of order two at a point $z \in C$, so that
  $\lambda$ equips $C - \{ z_0 \}$ with the structure of an affine
  line.  Denote by $\Sigma(C,z_0)$ the space of scalings on $C$ with
  pole at $z_0$, and by $\ol{\Sigma}(C,z_0)$ the compactification
  $\ol{\Sigma}(C,z_0) = \Sigma(C,z_0) \cup \{ 0, \infty \}$.  
\item A {\em $n$-marked scaled line} is a smooth projective curve of
  genus zero equipped with a non-degenerate scaling $\lambda \in
  \Sigma(C,z_0)$ and a collection $z_1,\ldots,z_n \in C$ of points
  distinct from each other and from the root marking $z_0$.
\end{enumerate}  
\end{definition}

Let $M_{n,1}(\bA)$ denote the moduli space of isomorphism classes of
$n$-marked scaled lines.  We may view $M_{n,1}(\bA)$ as the moduli
space of isomorphism classes of $n$-markings on an affine line $\bA$,
where two sets of markings are isomorphic if they are related by
translation.  ${M}_{n,1}(\bA)$ admits a compactification by allowing
nodal curves with possible degenerate scalings as follows.  

\begin{definition} \label{nodalscalings}
\begin{enumerate}
\item  {\rm (Dualizing sheaf and its projectivization)}  
Recall from e.g. \cite[p.91]{ar:alg2} that if $C$ is a genus zero
nodal curve then the dualizing sheaf $\omega_C$ on $C$ is locally free
of rank one, that is, a line bundle.  Explicitly, if $\ti{C}$ denotes
the normalization of $C$ (the disjoint union of the irreducible
components of $C$) with nodal points $\{ \{w_1^+,w_1^- \}, \ldots, \{
w_k^+, w_k^- \} \}$ then $\omega_C$ is the sheaf of sections of
$\omega_{\ti{C}} := T^\dual \ti{C}$ whose residues at the points $
w_j^+, w_j^-$ sum to zero, for $j = 1,\ldots,k$.  Denote by
$\P(\omega_C \oplus \C)$ the fiber bundle obtained by adding in a
section at infinity.
\item {\rm (Scalings on nodal curves)} Let $C$ be a connected
  projective nodal curve of arithmetic genus zero.  A {\em scaling} on
  $C$ is a section $\lambda : C \to \P(\omega_C \oplus \C)$ such that
  the restriction of $\lambda$ to any irreducible component is a
  (possibly degenerate) scaling as in Definition \ref{smoothscalings}.
\item {\rm (Scaled affine lines)} A {\em nodal $n$-marked scaled line}
  consists of
\begin{enumerate}
\item a connected projective nodal curve $C$ of arithmetic genus zero,
\item a scaling $\lambda : C \to \P(\omega_C \oplus \C)$
\item a collection of markings $z_0,\ldots,z_n \in C$ distinct from
  the nodes
\end{enumerate}
such that the following {\em monotonicity condition} is satisfied:
\begin{enumerate} 
\item[(i)] for each $i = 1,\ldots, n$, there is exactly one
  irreducible component of $C_{+,i}$ of $C$ on the path of irreducible
  components between $z_0$ and $z_i$ on which $\lambda$ is finite and
  non-zero, with double pole at the node which disconnects the
  component from the root marking $z_0$, 
\ccomment{slightly changed}
and
\item[(ii)] the irreducible components other than $C_{+,i}$ on the path of irreducible components
  between $z_i$ and $z_0$ have either $\lambda = 0$ (if they can be
  connected to $z_i$ without passing through $C_{+,i}$) or $\lambda =
  \infty$ (if they are connected to $z_0$ without passing through
  $C_{+,i}$).
\end{enumerate}
A nodal marked scaled affine line is {\em stable} if each irreducible
component with non-degenerate scaling has at least two special points,
and each irreducible component with degenerate scaling has at least
three special points. 
\item {\rm (Combinatorial types of scaled affine lines)} The {\em
  combinatorial type} of a nodal scaled affine line is the {\em rooted
  colored tree} $\Gamma = (\Ve(\Gamma),\Edge(\Gamma))$ whose vertices
  are the irreducible components of $C$, edges are the nodes and
  markings, equipped with a bijection from the set of semi-infinite
  edges $\Edge_\infty(\Gamma)$ to $\{ 0 ,\ldots, n \}$ given by the
  markings, and a subset of {\em colored vertices} $\Ve^+(\Gamma)
  \subset \Ve(\Gamma)$ corresponding to irreducible components with
  non-degenerate scalings.  This ends the definition.
\end{enumerate}  \end{definition} 

\begin{example}  See Figure \ref{zilt} below for an example of a nodal scaled affine line, where irreducible
components with $\lambda = 0$ resp. $\lambda$ finite and non-zero
resp. $\lambda$ is infinite are shown with light resp. medium
resp. dark shading.  The example shown is not stable, because several
of the lightly shaded components and darkly shaded components have
less than three special points.
\end{example}  

\begin{figure}[ht]
\begin{picture}(0,0)%
\includegraphics{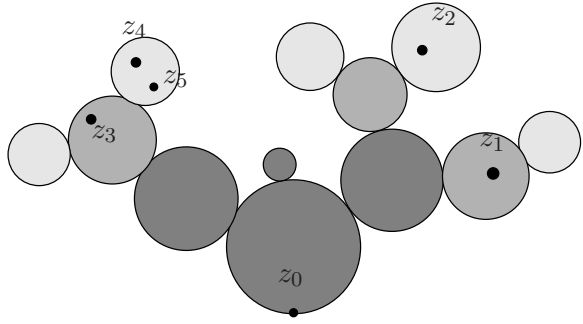}%
\end{picture}%
\setlength{\unitlength}{4144sp}%
\begingroup\makeatletter\ifx\SetFigFontNFSS\undefined%
\gdef\SetFigFontNFSS#1#2#3#4#5{%
  \reset@font\fontsize{#1}{#2pt}%
  \fontfamily{#3}\fontseries{#4}\fontshape{#5}%
  \selectfont}%
\fi\endgroup%
\begin{picture}(3441,1893)(2013,-3625)
\put(3635,-3434){\makebox(0,0)[lb]{{{{$z_0$}%
}}}}
\put(4840,-2624){\makebox(0,0)[lb]{{{{$z_1$}%
}}}}
\put(4552,-1868){\makebox(0,0)[lb]{{{{$z_2$}%
}}}}
\put(2518,-2571){\makebox(0,0)[lb]{{{{$z_3$}%
}}}}
\put(2699,-1948){\makebox(0,0)[lb]{{{{$z_4$}%
}}}}
\put(2946,-2243){\makebox(0,0)[lb]{{{{$z_5$}%
}}}}
\end{picture}%
\caption{An nodal scaled line}
\label{zilt}
\end{figure} 

\begin{remark}  {\rm (Affine structures on the components with non-degenerate
scalings)} The monotonicity condition implies that the restriction of
  $\lambda$ to any irreducible component $C_{i,+}$ has a unique pole,
  hence a unique double pole at the nodal point $\hat{z}_i$ connecting
  $C_{i,+}$ with the component containing $z_0$, and so defines an
  affine structure on the complement $C_{i,+} - \hat{z}_i$.  The other
  components have no canonical affine structures.  \ccomment{changed
    slightly}
\end{remark} 

Let $\M_{n,1,\Gamma}(\bA)$ resp. $\ol{M}_{n,1}(\bA)$ denote the moduli
space of isomorphism classes of stable scaled $n$-marked affine lines
of type $\Gamma$ resp. the union over combinatorial types.  We call
$\ol{M}_{n,1}(\bA)$ the {\em complexified multiplihedron}.

\begin{proposition}  \cite{mau:mult} The spaces $\ol{M}_{n,1}(\bA)$ 
 admit the structure of quasiprojective resp. projective varieties of
 dimension
\begin{equation} \label{dimform} 
 \dim(M_{n,1,\Gamma}(\bA)) = n - 2 - |\Edge_{< \infty}(\Gamma)| +
 |\Ve^+(\Gamma)|, \ \ \ \dim(\ol{M}_{n,1}(\bA)) = n - 1
 .\end{equation}
\end{proposition} 

The space $\ol{M}_{n,1}(\bA)$ was first studied in Ziltener's thesis
\cite{zilt:phd} in the context of gauged Gromov-Witten theory on the
affine line.

\begin{example}  {\rm (The second complexified multiplihedron)}  
The moduli space $\ol{M}_{n,1}(\bA)$ in the first non-trivial case $n
= 2$ admits an isomorphism
\begin{equation} \label{diff} \ol{M}_{2,1}(\bA) \to \P, \ \ \ [z_1,z_2 ] \mapsto z_1 - z_2 
\end{equation}
(here $\P$ is the projective line) with two distinguished points given
by nodal scaled affine lines, appearing in the limit where the two
markings become infinitely close or far apart, see Figures
\ref{twoconverge}, \ref{twodiverge}.  Here $[z_1,z_2] \in
M_{2,1}(\bA)$ is a point in the open stratum, given by markings at
$z_1,z_2$ modulo translation only on $\bA$.
\end{example} 

\begin{figure}[ht]
\begin{picture}(0,0)%
\includegraphics{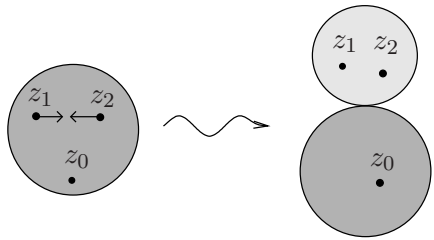}%
\end{picture}%
\setlength{\unitlength}{4144sp}%
\begingroup\makeatletter\ifx\SetFigFont\undefined%
\gdef\SetFigFont#1#2#3#4#5{%
  \reset@font\fontsize{#1}{#2pt}%
  \fontfamily{#3}\fontseries{#4}\fontshape{#5}%
  \selectfont}%
\fi\endgroup%
\begin{picture}(2547,1399)(3470,-3678)
\put(5643,-3271){\makebox(0,0)[lb]{{{{$z_0$}%
}}}}
\put(3814,-3259){\makebox(0,0)[lb]{{{{$z_0$}%
}}}}
\put(3595,-2886){\makebox(0,0)[lb]{{{{$z_1$}%
}}}}
\put(3969,-2894){\makebox(0,0)[lb]{{{{$z_2$}%
}}}}
\put(5419,-2566){\makebox(0,0)[lb]{{{{$z_1$}%
}}}}
\put(5665,-2574){\makebox(0,0)[lb]{{{{$z_2$}%
}}}}
\end{picture}%
\caption{Two markings converging} 
\label{twoconverge}
\end{figure}

\begin{figure}[ht]
\begin{picture}(0,0)%
\includegraphics{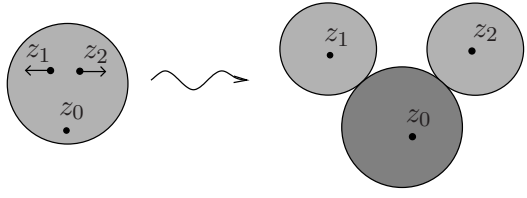}%
\end{picture}%
\setlength{\unitlength}{4144sp}%
\begingroup\makeatletter\ifx\SetFigFont\undefined%
\gdef\SetFigFont#1#2#3#4#5{%
  \reset@font\fontsize{#1}{#2pt}%
  \fontfamily{#3}\fontseries{#4}\fontshape{#5}%
  \selectfont}%
\fi\endgroup%
\begin{picture}(3105,1120)(3468,-3708)
\put(3787,-3290){\makebox(0,0)[lb]{{{{$z_0$}%
}}}}
\put(3585,-2946){\makebox(0,0)[lb]{{{{$z_1$}%
}}}}
\put(3930,-2953){\makebox(0,0)[lb]{{{{$z_2$}%
}}}}
\put(5861,-3332){\makebox(0,0)[lb]{{{{$z_0$}%
}}}}
\put(5372,-2851){\makebox(0,0)[lb]{{{{$z_1$}%
}}}}
\put(6261,-2826){\makebox(0,0)[lb]{{{{$z_2$}%
}}}}
\end{picture}%
\caption{Two markings diverging} 
\label{twodiverge}
\end{figure}

\begin{remark} {\rm (Embedding via forgetful morphisms)}  
More generally, for arbitrary $n$ there exists for any choice of
$\{i,j \} \subset \{1,\ldots, n \}$ of subset of order $2$ a {\em
  forgetful morphism}
$$ f_{i,j} : \ol{M}_{n,1}(\bA) \to \ol{M}_{2,1}(\bA) $$
forgetting the markings other than $i,j$ and collapsing all unstable
irreducible components, and for any choice $\{ i,j,k,l \} \subset \{ 0,\ldots, n \}$
of subset of order $4$ a forgetful morphism 
$$ f_{i,j,k,l} : \ol{M}_{n,1}(\bA) \to \ol{M}_4 $$ 
given by forgetting the scaling and all markings except $i,j,k,l$, and
collapsing all unstable irreducible components.  The product of
forgetful morphisms defines an embedding into a product of projective
lines.
\end{remark}  

The variety $\ol{M}_{n,1}(\bA)$ is not smooth, but rather has toric
singularities, see Section \ref{local}.  The {\em boundary divisors}
are the closures of strata $M_{n,1,\Gamma}$ of codimension one. 

\begin{remark} {\rm (Description of the boundary divisors of the complexified multiplihedron)} 
 From the dimension formula \eqref{dimform} one sees that there are
 two types of boundary divisors.  First, for any $I \subset
 \{1,\ldots, n \}$ with $|I| \ge 2$ we have a divisor
$$\iota_I: D_I \to \ol{M}_{n,1}(\bA)$$
corresponding to the formation of a single bubble containing the
markings $I$.  This divisor admits a gluing isomorphism
\begin{equation} \label{homeo1} 
 D_I \to \ol{M}_{|I| + 1} \times \ol{M}_{n - |I|+1,1}(\bA) .\end{equation}
Call these divisors of {\em type} I.  Second, for any unordered partition $I_1
\cup \ldots \cup I_r $ of $\{1 ,\ldots, n\}$ with $r \ge 2$ we have a
divisor $D_{I_1,\ldots,I_r}$ corresponding to the formation
of $r$ bubbles with markings $I_1,\ldots, I_r$, attached to a
remaining irreducible component with infinite scaling.  This divisor admits a
gluing isomorphism
\begin{equation} \label{homeo2}
 D_{I_1,\ldots,I_r
 } \cong \left( \prod_{i=1}^r \ol{M}_{|I_i|,1}(\bA) \right) \times \ol{M}_{r+1}
 .\end{equation}
Call these {\em divisors of type} II.  Note that the divisors of type
I and type II roughly correspond to the terms in the definition of
$A_\infty$ functor.
\end{remark} 

Recall that a {\em Weil divisor} on a normal scheme $X$ is a formal,
locally finite sum of codimension one subvarieties, while a {\em
  Cartier divisor} is a Weil divisor given as the zero set of a
meromorphic section of a line bundle with multiplicities given by the
order of vanishing of the section \cite[Remark 6.11.2]{ha:al}.
\ccomment{made more precise} For smooth varieties, any Weil divisor is
Cartier.  Since $\ol{M}_{n,1}(\bA)$ is not smooth, Weil divisors are
not necessarily Cartier, in particular, a Weil divisor may not admit a
dual cohomology class of degree $2$.  That is, for a Weil divisor
\begin{equation} \label{coeff}
 D = \sum_I n_I [D_I] + \sum_{I_1 \sqcup \ldots \sqcup I_r = \{ 1,\ldots
,n \} } n_{I_1, \ldots, I_r} [D_{I_1 , \ldots , I_r }]
\end{equation}
there may or may not exist a class $\delta \in
H^2(\ol{M}_{n,1}(\bA))$ that satisfies
$$ \langle \beta, [D] \rangle = \langle \beta \wedge \delta ,
[\ol{M}_{n,1}(\bA)] \rangle .$$

Let $(V,(\mu_V^n)_{n \ge 2}) $ and $(W,(\mu_W^n)_{n \ge 2}) $ be CohFT
algebras over a $\Q$-ring $\Lambda$.  

\begin{definition} {\rm (Morphisms of CohFT algebras)}   A {\em morphism of
CohFT algebras} from $V$ to $W$ is a collection of $S_n$-invariant,
  multilinear maps \ccomment{added $\Lambda$}
$$ \phi^n: V^{n} \times H(\ol{M}_{n,1}(\bA)) \to W, \quad n \ge 0 $$
such that for any Cartier divisor $D$ of the form \eqref{coeff} with
dual class $\delta \in H^2(\ol{M}_{n,1}(\bA))$, any $v \in V^n$
and any $\beta \in H(\ol{M}_{n,1}(\bA))$
%
%
\begin{multline} \label{weaksplit}
\phi^n(v, \beta \wedge \delta) = 
\sum_{I} n_I \phi^{n - |I| + 1} ( \mu_V^{|I|}(v_i, i \in I; \cdot )
,v_j, j \notin I; \cdot)(\iota_I^* \beta) \\ 
+  \sum_{r \leq s, I_1,\ldots,I_r} n_{I_1,\ldots,I_r} \mu_W^s( \phi^{|I_1|}( v_i, i \in I_1; \cdot), \ldots,
\\ \phi^{|I_r|}( v_i, i \in I_r; \cdot),\phi^0(1),\ldots,
\phi^0(1); \cdot ) (\iota_{I_1,\ldots,I_r}^* \beta)/(s-r)!
\end{multline}
where $\cdot$ indicates insertion of the Kunneth components of
$\iota_I^* \beta$, $\iota_{I_1,\ldots, I_r}^* \beta$, using the
homeomorphisms \eqref{homeo1}, \eqref{homeo2} and the sum on the
right-hand side is, by assumption, finite.  The element $\phi^0(1) \in
W$ is the {\em curvature} of the morphism $(\phi^n)_{n \ge 0}$, and
$(\phi^n)_{n \ge 0}$ is {\em flat} if the curvature vanishes.  A {\em
  morphism of filtered CohFT algebras $V,W$} is a collection of maps
$\phi^n$ as above such that each $\phi^n$ preserves the filtrations in
the sense that $\phi^n$ maps $V_a^n \times H(\ol{M}_{n,1}(\bA))$ to
$W_a$ and \eqref{weaksplit} is finite modulo $W_a$ for any $a \in \R$.
\end{definition}

\begin{example}  $\ol{M}_{2,1}(\bA) \cong \P$ and so every Weil divisor
is Cartier and any two prime Weil divisors are linearly equivalent.
In particular, the equivalence $ [D_{ \{ 1,2 \} }] = [D_{ \{ 1 \}, \{
    2 \}}] $ holds in $H^2(\ol{M}_{2,1}(\bA),\Q) = \Q$.  Hence if
$(\phi^n)_{n \ge 0}: (V,(\mu_V^n)_{n \ge 2}) \to (W, (\mu_W^n)_{n\ge
  2}) $ is a flat morphism of CohFT algebras then $\phi^1:V \to W$ is
a homomorphism, $ \phi^1 \circ \mu^2_V = \mu^2_W \circ (\phi^1 \times
\phi^1) .$
\end{example} 

Recall that the notion of CohFT may be reformulated as a Frobenius
manifold structure of Dubrovin \cite{dub:geom}, consisting of a datum
$(V,g,F,1,e)$ of an affine manifold $V$, a metric $g$ on the tangent
spaces, a potential $F$ whose third derivatives provide the tangent
spaces $T_v V$ with associative multiplications $\star_v: T_v^2V \to
T_v V$, a unity vector field $1$ and an Euler vector field $e$
providing a grading.  Any CohFT $(V, (\mu^n)_{n \ge 2})$ defines a
formal Frobenius manifold \cite{man:fro} with formally associative
products
\begin{equation} \label{starv}
 \star_v: T_v^2V \to T_v V ,\quad (w_1,w_2) \mapsto \sum_{n \ge 0}
 \mu^{n+2}(w_1,w_2,v,\ldots,v)/n! .\end{equation}
Formal associativity means that the Taylor coefficients in the
expansion of
$$ (w_1 \star_v w_2) \star_v w_3 - w_1 \star_v (w_2 \star_v w_3) $$
around $v = 0$ vanish to all orders for any $w_1,w_2,w_3 \in T_v V$;
in good cases one has convergence of the corresponding infinite sums.
Later, a weaker notion of {\em $F$-manifold} was introduced by Manin
and Hertling \cite{her:weak}, which consists of a pair $(V,\circ)$
where $\circ$ is a family of multiplications on the tangent spaces
$T_vV$ satisfying a certain axiom. In other words, one forgets the
data $g,1,e$.  This weaker notion is compatible with the notion of
morphism of CohFT algebras:

\begin{proposition}  {\rm (Algebra homomorphisms on tangent spaces)}  
Any morphism of CohFT algebras $( \phi^n )_{n \ge 0}$ from $V$ to $W$
defines a formal map
$$ \phi: V \to W, \quad
v \mapsto \sum_{n \ge 0} \frac{1}{n!} \phi^n(v,\ldots, v;1) $$
with the property that for any $v \in V$ \ccomment{added any v}
the linearization
$ D_v \phi: T_v V \to T_{\phi(v)} W $
is a $\star$-homomorphism in the sense that
\begin{equation} \label{starhom} D_v \phi(w_1) \star_{\phi(v)} D_v \phi(w_2) = D_v \phi( w_1 \star_v w_2) ,
\ \forall w_1,w_2 \in T_v V. \end{equation}
\end{proposition}  
 
\noindent By a formal map we mean a map from a formal neighborhood of
$0$ in $V$ to a formal neighborhood of $\phi(0)$ in $W$.  The equation
\eqref{starhom} holds in the sense of Taylor expansion around $v = 0$
to all order.

\begin{proof}  Consider 
the divisor relation $D_{ \{ 1,2 \} } \sim D_{ \{ 1 \}, \{ 2 \}}$ on
$\ol{M}_{2,1}(\bA)$.  Its pull-back to $\ol{M}_{n,1}(\bA)$ is the
relation
\begin{equation} \label{bigrelation} 
 \sum_{I_1 \ni 1,I_2 \ni 2,I_3,\ldots,I_r} D_{I_1,I_2,\ldots, I_r}
 \sim \sum_{I \supset \{ 1, 2 \}} D_I \end{equation}
where the first sum is over unordered partitions $I_1,\ldots, I_r$
with $1 \in I_1,2 \in I_2$ and each $I_j, j = 1,\ldots, r$ non-empty,
and the second is over subsets $I \subset \{1,\ldots, n \}$ with $\{
1,2 \} \subset I$.  Indeed, the map \eqref{diff} composes with the
forgetful map to give a rational function
$$ f_{2,1}: \ol{M}_{n,1}(\bA) \to \ol{M}_{2,1}(\bA) \cong \P .$$
For $n = 2$, this map identifies $D_{ \{ 1,2 \}} \to \{ 0 \}, \ D_{ \{
  1 \}, \{ 2 \}} \to \{ \infty \} .$ For arbitrary $n$, one checks
using the charts in Ma'u-Woodward \cite{mau:mult} that the order of
vanishing of $f_{2,1}$ on $D_I$ resp. $D_{I_1,\ldots,I_r}$ is $1$
resp. $-1$ iff $ \{ 1, 2 \} \subset I$ resp. $I_1,I_2$ separate $1,2$
and $0$ otherwise.  Since the partitions are unordered, if $I_1,I_2$
separate $1,2$ we may assume that $1 \in I_1$ resp.  $2 \in I_2$.
Note that the number of ways of choosing the partition on the
left-hand-side of \eqref{bigrelation} with sizes $i_1,\ldots,i_r$ is $
\left( \begin{array}{c} n- 2 \\ i_1 - 1 \ i_2 - 1 \ i_3 \ldots
  i_r \end{array} \right) $.  We compute
\begin{eqnarray*}
D_v \phi(w_1 \star_v w_2) 
&=&\sum_{n,i}\frac{1}{(i-2)!(n-i)!}\phi^{n-i+1}\big(\mu^i_V(w_1,w_2,v,\ldots,v;1),v,\ldots,v;1\big)\\
&=&\sum_{n,I} ((n-2)!)^{-1} \phi^{n-|I|+1}\big(\mu^{|I|}_V(w_1,w_2,v,\ldots,v;1),v,\ldots,v;1\big)\\
&=& \sum_{I_1 \ni 1, I_2 \ni 2, I_3,\ldots,I_r} 
((n-2)! \# \{j \, | \, I_j = \emptyset \} !)^{-1} 
\mu^r_W \big(
\phi^{|I_1|}(w_1,v,\ldots,v;1),\\
&&\phi^{|I_2|}(w_2,v,\ldots,v;1),\phi^{|I_3|}(v,\ldots,v;1),\ldots,\phi^{|I_r|}(v,\ldots,v;1);1\big) \\
&=& \sum_{i_1,i_2 \ge 1,i_3 \ldots,i_r\geq 0}\frac{1}{(i_1-1)!(i_2-1)!i_3!\cdots i_r!(r-2)!}\mu_W^r\big(\phi^{i_1}(w_1,v,\ldots,v;1),\\
&&\phi^{i_2}(w_2,v,\ldots,v;1),\phi^{i_3}(v,\ldots,v;1),\ldots,\phi^{i_r}(v,\ldots,v;1);1\big) \\
&=&\sum_{r}
\frac{1}{(r-2)!}\mu_W^r\big(D_v \phi(w_1),D_v \phi(w_2),\phi(v),\ldots, \phi(v)) \\ 
&=& D_v \phi(w_1) \star_{\phi(v)} D_v \phi(w_2) 
\end{eqnarray*}
where the right-hand-side is assumed to be a finite sum (modulo any
$W_a$, for a morphism of filtered CohFT algebras).  Here the first
equality is by definition of $\phi,\star_v$ and the second replaces
the sum over $i$ with the sum over subsets $I$ containing $1,2$. The
third follows from the splitting axiom \eqref{weaksplit}, where the
elements of the partition $I_1,\ldots,I_r$ may be empty.  The fourth
equality replaces the sum over unordered partitions $I_1,\ldots,I_r$
with $I_1 \ni 1, I_2 \ni 2$ with the sum over their sizes $i_1,\ldots,
i_r$, with the additional factorial $(r-2)!$ arising from the possible
orderings of the subsets $I_3,\ldots,I_r$.  The fifth equality follows
by definition of $\phi, \mu_W$, and the last equality follows by
definition of $\star_{\phi(v)}$.
\end{proof}  

\begin{remark} 
It would be interesting to characterize which $\star$-morphisms arise
from morphisms of CohFT algebras.  This would require a study of the
cohomology ring of $\ol{M}_{n,1}(\bA)$ along the lines of Keel
\cite{keel:int} for the moduli space of stable marked genus zero
curves; this paper is essentially a partial study of the second
cohomology group only.  The most naive possibility would be an analog
of Keel's result \cite{keel:int}, namely that $H(\ol{M}_{n,1}(\bA))$
is generated by the classes of the Cartier boundary divisors modulo
the relations given by the preimages of $D_{\{1 \} \{ 2 \}} - D_{ \{
  1,2 \}}$ under the forgetful morphism $f_{ij}: \ol{M}_{n,1}(\bA))
\to \ol{M}_{2,1}(\bA)$, as $i,j$ range over distinct elements of
$\{1,\ldots, n \}$, and the products $D' D''$, if $D'$ and $D''$ are
disjoint Cartier divisors.
\end{remark} 

\section{Quantum Kirwan morphism} 

In this section we describe the motivating example for the theory of
morphisms of CohFT algebras in the previous section, the quantum
Kirwan morphism.  Let $G$ be a compact Lie group, $G_\C$ its
complexification, and $X$ be a smooth projective $G_\C$-variety
equipped with a polarization (ample $G$-line bundle) \ccomment{added
  defn of polarization} such that $G$ acts locally freely on the
semistable locus.  The classical Kirwan morphism $ H_G(X) \to H(X \qu
G) $ is surjective by Kirwan's thesis \cite{ki:coh}.  Computing the
kernel of the Kirwan morphism therefore gives a presentation of the
cohomology ring of the quotient $X \qu G$.  Let $QH_G(X)$ resp. $QH(X
\qu G)$ denote the corresponding quantum cohomologies defined over the
universal Novikov field.  Each has the structure of a CohFT algebra,
with products given by suitable counts of genus zero stable maps.  The
quantum version of the Kirwan morphism is a morphism of CohFT algebras
$$Q \kappa : QH_G(X) \to QH(X \qu G).$$
The virtual fundamental cycles are constructed algebraically in
\cite{qkirwan}.  We describe first the symplectic approach.

\subsection{Affine vortices} 

From the symplectic point of view the quantum Kirwan morphism is
defined by a count of {\em affine vortices} introduced in Ziltener
\cite{zilt:phd}, \cite{zilt:qk2}.  There is also an algebro-geometric
interpretation, as a count of certain morpisms to the quotient stack
$X/G_\C$, that we present later.  Let $\g$ denote the Lie algebra of
$G$, and let $\Phi: X \to \g^\dual$ be a moment map for the action of
$G$ on $X$ arising from a unitary connection on the polarization.  For
any connection $A \in \Omega^1(\bA,\g)$, we denote by $F_A \in
\Omega^2(\bA,\g)$ its curvature.  We assume that $\g$ is equipped with
an invariant metric inducing an identification $\g \to \g^\dual$.

\begin{definition}  {\rm (Affine symplectic vortices)} An {\em $n$-marked affine symplectic vortex} 
to $X$ is a datum $(A,u,\ul{z})$, where $A \in \Omega^1(\bA,\g)$ is a
connection on the trivial bundle, $u: \bA \to X$ is a holomorphic with
respect to the complex structure determined by $A$, $\ul{z} =
(z_1,\ldots, z_n) \in \bA^n$ is a collection of distinct points, and
$$ F_A + u^* \Phi \, \Vol_\bA = 0. $$
Here $\Vol_\bA = \frac{i}{2} \d z \wedge \d \ol{z}$ is the standard
real area form on $\bA$.  

An {\em isomorphism} of marked symplectic vortices $(A_j, u_j,
\ul{z}_j),j = 0,1 $ is an automorphism of the trivial bundle $\phi: \bA
\times G \to \bA \times G$ such that $\phi^* A_1 = A_0$ and $\phi^* u_1 =
u_0$ (thinking of $u_0,u_1$ as sections of the associated $X$-bundle)
such that $\phi$ covers a translation on the base, that is, there
exists a $\tau \in \C$ such that $ \pi \circ \phi( z,g) = z + \tau$
for all $z,g \in \bA \times G$, and $z_{i,1} = z_{i,0} + \tau$ for $i
= 1,\ldots, n$.

The {\em energy} of a vortex $(A,u,\ul{z})$ is given by
\begin{equation} \label{energy} E(A,u) = \hh \int_{\bA}   \left( \Vert \d_A u \Vert^2 + \Vert F_A\Vert^2 + \Vert u^*\Phi \Vert^2 \right) \Vol_{\bA}.\end{equation}
This ends the definition.
\end{definition} 

Let $M_{n,1}^G(\bA,X)$ denote the moduli space of isomorphism classes
of finite energy $n$-marked vortices on $\bA$ with values in $X$.  The
following Hitchin-Kobayashi correspondence gives an algebro-geometric
description of the moduli space of affine vortices.  Its proof is part
of a joint project with Venugopalan to appear elsewhere, see
\cite{venu:heat}.  By definition \cite{dm:irr} a morphism $u$ from the
projective line $\P$ to the quotient stack $X/G_\C$ consists of a
$G_\C$-bundle $P \to \P$ together with a section $\P \to P
\times_{G_\C} X$.  By the git quotient $X \qu G_\C$ we mean the
stack-theoretic quotient of the semi-stable locus by the group action;
if stable=semistable then $X \qu G_\C$ has coarse moduli space the
projective variety considered in Mumford et al
\cite{mu:ge}. 

\begin{theorem} {\rm (Classification of affine vortices)} 
\label{vw}
 Suppose that $X$ is a smooth polarized projective $G_\C$-variety such
 that $G_\C$ acts freely on the semistable locus of $X$.  There is a
 bijection between isomorphism classes of finite energy affine
 vortices and isomorphism classes of morphisms $u$ from the projective
 line $\P$ to the quotient stack $X/G_\C$ such that $u(\infty)$ lies
 in the semistable locus $X \qu G_\C \subset X/G_\C$.
\end{theorem}    

The moduli space $M_{n,1}^G(\bA,X)$ admits a compactification
$\ol{M}_{n,1}^G(\bA,X)$ allowing nodal scaled lines as the domain:

\begin{definition} {\rm (Affine scaled gauged maps)}  An  {\em affine marked nodal scaled gauged map} 
to $X$ is a marked nodal scaled line $(C,\lambda,\ul{z})$ together
with a morphism $u: C \to X/G_\C$ such that
\begin{enumerate}
\item {\rm (Semistable bundle where the scaling is zero)} for each
  irreducible component $C_i$ with zero scaling $\lambda | C_i = 0$,
  the $G$-bundle on $C_i$ is semistable, hence trivializable;
\item {\rm (Semistable point where the scaling is infinite)} for each $z \in
  C$ with $\lambda(z) = \infty$, the image $u(z)$ lies in the
  semistable locus $X \qu G_\C$.
\end{enumerate}
A nodal scaled morphism is {\em stable} if it has no infinitesimal
automorphisms, or equivalently, if each irreducible component on which $u$ is
trivial has at least three special points, or two special points and
non-degenerate scaling.    This ends the definition.
\end{definition} 

\begin{remark} 
\begin{enumerate} 
\item {\rm (Evaluation and forgetful morphisms)} 
Let $\ol{M}_{n,1}^G(\bA,X)$ denote the moduli space of isomorphism
classes of stable nodal scaled maps to $X$.  $\ol{M}_{n,1}^{G}(\bA,X)$
admits an evaluation map at the markings, and if the action of $G$ on
the semistable locus is free, an additional evaluation map at infinity
to $X \qu G$ \cite{zilt:phd}, \cite{zilt:qk2}:
$$ \ev \times \ev_\infty: \ol{M}_{n,1}^G(\bA,X) \to (X/G_\C)^{n} \times X
\qu G_\C .$$
For $n > 0$, there is a forgetful morphism to the moduli space of
scaled lines,
$$f:\ol{M}_{n,1}^G(\bA,X) \to \ol{M}_{n,1}^G(\bA) .$$
\item {\rm (The locally free case)} In the case that $G$ acts only
  {\em locally} freely on the semistable locus in $X$, the quotient $X
  \qu G$ is an orbifold or smooth {\em Deligne-Mumford stack}.  The
  Hitchin-Kobayashi correspondence in this case relates affine
  vortices to representable morphisms of a {\em weighted projective
    line} $\P(1,r) \to X/G_\C$ for some $r > 0$ \ccomment{changed
    $\ge$ to $>$} such that $\infty$ maps to the semistable locus, so
  that the evaluation map at infinity
$$ \ev_\infty: \ol{M}_{n,1}^G(\bA,X) \to \ol{I}_{X \qu G_\C} $$
takes values in the {\em rigidified inertia stack}
$$\ol{I}_{X \qu G_\C} := \cup_{r \ge 1} \Hom^{\on{rep}}( \P(r), X \qu
G_\C)/\P(r) $$
of representable morphisms from $\P(r) = B\Z_r$ to $X \qu G_\C$ modulo
$\P(r)$, for some integer $r \ge 1$.  \ccomment{changed to $r \ge 1$,
  not $r \ge 0$}
See
Abramovich-Olsson-Vistoli \cite{aov:twisted} and
Abramovich-Graber-Vistoli \cite{agv:gw} for more on the definition of
$\ol{I}_{X \qu G}$.
\end{enumerate} 
\end{remark} 

\subsection{Quantum Kirwan morphism} 

The quantum Kirwan map is defined by virtual integration over the
moduli space of affine vortices introduced in the previous subsection.
Existence and axiomatic properties of virtual fundamental classes for
the case of smooth projective varieties as target using the
Behrend-Fantechi \cite{bf:in} machinery are proved in \cite{qkirwan}.
Some results in the direction of establishing the existence of
fundamental classes for target compact Hamiltonian actions were taken
in \cite{zilt:fred}.  Here we review the case of algebraic target.

\begin{definition}  \label{novikov} 
\begin{enumerate}
\item {\rm (Novikov field)} Given a smooth projective $G_\C$-variety
  $X$ and an equivariant symplectic class $[\omega_G] \in H_2^G(X)$
  define the {\em Novikov field} $\Lambda_X^G$ for $X$ as the set of
  all maps $\lambda: H^G_2(X):=H^G_2(X,\Q)\to \Q$ such that for every
  constant $c$, the set of classes
$$ \left\{ d \in H^G_2(X,\Q), \ \lan [\omega], d \ran \leq c
\right\} $$
on which $\lambda$ is non-vanishing is finite.  The delta function at $d$ is
denoted $q^d$.  Addition is defined in the usual way and
multiplication is convolution, so that $q^{d_1} q^{d_2} = q^{d_1 +
  d_2}$. 
\item {\rm (Equivariant quantum cohomology)} Define as vector spaces
$$ QH^G(X,\Q) := H^G(X,\Q) \otimes \Lambda_X^G .$$
Let $QH(X \qu G)$ denote the quantum cohomology defined over the
Novikov field $\Lambda_X^G$, that is,
$$QH(X \qu G) := H(\ol{I}_{X \qu G},\Q) \otimes \Lambda_X^G .$$  
\item {\rm (Quantum Kirwan morphism)} For each $n \ge 0$ define a map
$$ Q\kappa^{n} : QH_G(X)^{n} \times H(\ol{M}_{n,1}(\bA)) \to QH(X \qu G
) $$
as follows.  For $\alpha \in H_G(X)^{n}, \beta \in
H^*(\ol{M}_{n,1}(\bA))$ let 
$$ (Q\kappa^{n} (\alpha,\beta), \alpha_\infty) = \sum_{d \in
  H_2^G(X,\Q)} q^d \int_{\ol{M}_{n,1}^G(\bA,X,d)} \ev^* \alpha \cup
f^* \beta \cup \ev^*_\infty \alpha_\infty $$
using Poincar\'e duality; the pairing on the left is given by cup
product and contraction with the fundamental class of $X \qu G$.
\end{enumerate}
\end{definition}

\begin{theorem} \cite{qkirwan} 
{\rm (Quantum Kirwan morphism)} Suppose that $X$ is a smooth polarized
projective $G$-variety such that $G_\C$ acts locally freely on the
semistable locus of $X$.  The collection $(Q\kappa^n )_{n \ge 0}$ is a
morphism of CohFT algebras from $QH_G(X)$ to $QH(X \qu G)$.  If $X$ is
$G$-Fano in the sense that $c_1^G$ is positive on all rational curves
to the quotient stack $X/G_\C$, \ccomment{added definition of
  $G$-Fano} then the curvature $ Q\kappa^0$ vanishes, so $(
Q\kappa^n)_{n > 0} $ is a flat morphism of CohFT algebras.
\end{theorem}

 In order to compute presentations of the quantum cohomology of $X \qu
 G$ one would like to know that the quantum analog of Kirwan's
 surjectivity theorem, namely that the linearization of map $QH_G(X)
 \to QH(X \qu G)$ at a generic point is surjective.  In the case of
 free quotients $X \qu G$, the conjecture follows from Kirwan's
 theorem and linearity over the Novikov ring, using a filtration
 argument.

\subsection{Quantum Kirwan for toric quotients} 

In this section we describe the quantum Kirwan map in the case that
$G$ is a torus acting on a vector space $X$, so that $X \qu G$ is a
toric orbifold.  We sketch a proof that the kernel of the
linearization of the quantum Kirwan map is Batyrev's quantization of
the Stanley-Reisner ideal associated to the toric fan.  This
reproduces for example the presentation of the quantum cohomology of
weighted projective planes described in Coates-Corti-Lee-Tseng
\cite{coates:wps}.

\begin{example} (Weighted Projective Line $\P(2,3)$
\cite{gw:surject}) Let $\C_2$ resp. $\C_3$ denote the weight space for
$G_\C = \C^\times$ with weight $2$ resp. $3$ so that $X = \C_2 \oplus \C_3$
and $X \qu G = \P(2,3)$.  Let $\theta_1$ resp. $\theta_2$
resp. $\theta_3$ resp. $\theta_3^2$ denote the generator of the
component of $QH(X \qu G) \cong H(\ol{I}_{X \qu G}) \otimes
\Lambda_X^G$ with trivial isotropy resp. $\Z_2$ isotropy
resp. corresponding to $\exp( \pm 2 \pi i /3) \in \Z_3$.  Let $\xi \in
H^2_G(X)$ denote the integral generator corresponding to the
representation with weight $1$.  Then we have the following table for $Q\kappa^1(\xi^k)$:
\begin{equation} 
 \label{exp1}
 \begin{array}{|c|c|c|c|c|c|c|} 
\hline 
   k                & 0 & 1 & 2 & 3 & 4  & 5 \\ 
\hline
Q\kappa^1(\xi^k)    & 1 & \theta_1 & q^{1/3} \theta_3/6 & q^{1/2} \theta_2/18 &  q^{2/3} \theta_3^2/36  &  q/108 \\
\hline 
\end{array} .\end{equation} 
Indeed, identifying $H_2^G(X,\Q) \cong \Q$ so that $H_2^G(X,\Z) \cong
\Z$ we see from Theorem \ref{vw} that \ccomment{replaced /G with $/G_\C$
in a number of places}
\begin{eqnarray*}
 M_{1,1}^G(\bA,X,0) &=& \{(a_0,b_0) \neq 0 \}/G_\C \cong \P(2,3) \\
 M_{1,1}^G(\bA,X,1/3) &=& \{ (a_0, b_1 z + b_0), b_1 \neq 0 \}/G_\C \cong
 \C^2/\Z_3 \\
 M_{1,1}^G(\bA,X,1/2) &=& \{ (a_1z + a_0, b_1 z + b_0 ), a_1 \neq 0 \}/G_\C \cong \C^3/\Z_2\\
 M_{1,1}^G(\bA,X,2/3) &=& \{ (a_1 z + a_0, b_2 z^2 + b_1 z + b_0), b_2 \neq 0 \}/G_\C 
\cong \C^4/\Z_3 \\
 M_{1,1}^G(\bA,X,1) &=& \{ (a_2 z^2 + a_1 z + a_0, b_3 z^3 + b_2 z^2 + b_1
 z + b_0), (a_2,b_3) \neq 0 \} /G_\C 
\end{eqnarray*}
The map
$$ \sigma: \ol{M}_{1,1}^{G,\fr}(\bA,X,1/3) \to \C_2 \oplus \C_3, \quad u
\mapsto u(0)  $$
defines a section with a single transverse zero, leading to the
integral
$$ \int_{\ol{M}_{1,1}^G(\bA,X,1/3)
} \ev_1^* 6\xi^2 = \int_{\ol{M}_{1,1}^G(\bA,X,1/3)}
\ev^*_1 \Eul(\C_2 \oplus \C_3) = 1/3 .$$
The pairing on the sector $H(\on{pt}/\Z_3) \otimes \Lambda_X^G$ in
$QH(\P(2,3))$ is defined by contraction with the orbifold fundamental
class, that is, $[\on{pt}]/3$, which cancels the factor of $1/3$ in
the integral above yielding the $k=3$ column.  (Put another way,
$Q\kappa_1(6 \xi^2)$ is the push-forward under
$\ol{M}_{1,1}^G(\bA,X,1/3) \to {\rm pt}/\Z_3$, whose fiber is $\C_2
\oplus \C_3$.)  The other integrals are similar.  From \eqref{exp1}
one sees that $Q \kappa^1$ is surjective with kernel $ \xi^5 - q/108$.
Hence
$$ QH( \P(2,3)) = \Lambda_X^G[\xi] / (\xi^5 - q/108) $$
where $\Lambda_X^G$ is the Novikov field of fractional powers of a single
formal variable $q$.  Note that the quantum
Kirwan map is not surjective in this case without inverting $q$, that
is, over the Novikov ring instead of the Novikov field, and that
although the Novikov field involves fractional powers of $q$, the
relations have only integer powers.
\end{example} 

More generally let $X$ be a vector space and $G$ a torus acting freely
so that $X \qu G$ is a proper Deligne-Mumford toric stack (orbifold).
We identify $G =U(1)^k$ and let $\rho_1,\ldots, \rho_k \in \g^\dual$
denote the weights of the action on $X$ with $\dim(X) = k$.  We also
identify $H_2^G(X,\Z)$ with the coweight lattice $\g_\Z =
\exp^{-1}(1)$ in the Lie algebra $\g$.

\begin{definition} [Quantum Stanley-Reisner Ideal] 
Let $QSR^G(X) \subset QH_G(X)$ be the {\em quantum Stanley-Reisner
  ideal}, generated by the elements for $d \in H_2^G(X,\Z)$
$$ \prod_{ \rho_j(d) \ge 0} \rho_j^{\rho_j(d)} -
q^d \prod_{ \rho_j(d) < 0} \rho_j^{- \rho_j(d)} .$$
\end{definition} 

Batyrev \cite{bat:qcr} in the case of smooth toric varieties
conjectured that the quantum cohomology $QH(X \qu G)$ has a
presentation
\begin{equation} \label{presenttoric}
{QH}^G(X)/ QSR^G(X) \cong QH(X \qu G),
\end{equation} 
This conjecture was proved for semipositive toric varieties by
Givental in \cite{gi:eq}, Cox-Katz \cite{ck:ms}, and is false in
general as pointed out by Spielberg \cite{sp:gw}, at least for the
obvious generators. Iritani \cite{iri:gmt} proved that any smooth
projective toric variety has quantum cohomology canonically isomorphic
to the Batyrev ring $QH_G(X)/ QSR^G(X)$, using corrected generators.
Coates-Corti-Lee-Tseng \cite{coates:wps} generalized the presentation
to the case of weighted projective spaces.

\begin{example} If $G_\C = \C^\times$ acts on $X = \C^2$ with weights
$a,b \in \Z$ so that $X \qu G$ is the weighted projective line
  $\P(a,b)$ then the quantum Stanley-Reisner ideal is generated by
  $(a\xi)^a (b\xi)^b - q$.  Then with our conventions the quantum
  cohomology of $\P(a,b)$ has generators $\xi$ and fractional powers
  of $q$, the single relation is $(a\xi)^a (b\xi)^b = q$,
  c.f. \cite{coates:wps}.
\end{example} 

 \ccomment{some of the $\kappa$'s were $\kappa_X^G$, I changed them}
\begin{theorem} 
\label{batyrev} {\rm (Orbifold Batyrev conjecture)}
\cite{gw:surject} After suitable completion, the linearization $D_0
Q\kappa$ is surjective and the kernel of $D_0 Q\kappa$ is the quantum
Stanley-Reisner ideal $QSR^G(X)$, so that $ T_{Q\kappa(0)} QH(X \qu G)
\cong T_0 QH^G(X)/QSR^G(X)$.
\end{theorem}

We give a partial proof by showing that for any $d \in H_2^G(X,\Z)$,
\begin{equation} \label{qsreq}
\int_{[\ol{M}_1^G(\bA,X,d)]} \Eul(\oplus_j
\C_{\rho_j}^{\rho_j(d)}) \cup \ev_\infty^* [\on{pt}] = 1 .\end{equation}
Let 
\begin{multline} \sigma: \ol{M}_1^G(\bA,X,d) \to \prod_{\rho_j(d) \ge 0}
\C_{\rho_j}^{\max(0,\rho_j(d))}, \quad (u,z) \mapsto ( u_j^{(k)}(z)
)_{k=1,\ldots,\rho_j(d),j = 1,\ldots,N} \end{multline}
denote the map constructed from the derivatives of the evaluation map
at the marking $z_1$.  The map $\sigma$ gives a transverse section
with a single zero on the locus $\ev_\infty^{-1}(\on{pt}) \subset
M_{1,1}^G(\bA,X,d)$ and the remaining factor $\Eul(\oplus_j
\C_{\rho_j}^{\min(0,\rho_j(d))})$ is the obstruction bundle.

We claim that $\sigma$ is non-vanishing on the boundary strata.  Let
$(C,\lambda,\ul{z},u)$ be a stable scaled map with reducible domain,
and let $d' \neq d$ denote the homology class of the irreducible
component containing $z_1$.  Since at least two irreducible components
have positive energy, $([\omega_G],d') < ([\omega_G],d)$.  By
assumption $X \qu G$ is non-empty, which implies that the symplectic
class $[\omega_G]$ can be written as a positive combination of the
weights $\rho_j$.  Hence $\rho_j(d') < \rho_j(d)$ for at least one
$j$.  Furthermore, among $j$ such that $\rho_j(d') < \rho_j(d)$ there
must exist at least one such that $u_j$ is non-zero.  Indeed the sum
of $\C_{\rho_j}$ with $\rho_j(d' - d) \ge 0$ is part of the unstable
locus in $X$, and \ccomment{and and -> and } so no morphism $u$ with
only those irreducible components non-zero can be generically
semistable.  The $\rho_j(d') + 1$-st derivative of $u_j$ is then a
non-zero constant, so $\sigma(u) \neq 0$.  The equation \eqref{qsreq}
follows.

\subsection{Composition of morphisms of CohFT algebras and reduction
in stages}
\label{composesec}

In this section we describe a notion of {\em composition} of morphisms
of CohFT algebras.  This will make CohFT algebras into an
infinity-category, whose higher morphisms are {\em commutative
  simplices} of CohFT algebras. This composition plays a natural role
in the {\em quantum reduction with stages} conjecture relating the
quantum Kirwan maps for $G/K$ and $K$ with that for $G$, when $K
\subset G$ is a normal subgroup.  The definition of composition of
morphisms of CohFT algebras involves a moduli space of {\em $s$-scaled
  $n$-marked} affine lines defined as follows.

\begin{definition}  {\rm (Multiply-scaled curves)} 
An {\em $s$-scaled, $n$-marked curve} is a datum
$(C,\ul{z},\ul{\lambda})$ where $(C,\ul{z})$ nodal marked curve and
$\ul{\lambda} = (\lambda_1,\ldots,\lambda_s)$ is an $s$-tuple of
scalings as in Definition \ref{nodalscalings} and in addition
satisfying the following {\em balanced condition}:
\begin{enumerate}
\item[] For each irreducible component $C_i$ of $C$ and any two
  scalings $\lambda_j,\lambda_k$ not both $0$ or both $\infty$,
  \ccomment{added any two scalings} the ratio $(\lambda_j |
  C_i)/(\lambda_k | C_i) \in \P$ is independent of the choice of
  $C_i$. 
\end{enumerate} 
An $s$-scaled, $n$-marked line $C$ is {\em stable} if each irreducible
component with at least one non-degenerate scaling has at least two
marked or nodal points, and each irreducible component with all
degenerate scalings has at least three marked or nodal points.  The
{\em combinatorial type} of an $s$-scaled, $n$-marked affine line $C$
is the tree whose vertices $\Ve(\Gamma)$ correspond to irreducible
components, finite edges $\Edge_{< \infty}(\Gamma)$ to nodes, and
equipped with a labelling of the semi-infinite edges
$\Edge_\infty(\Gamma)$ by $\{ 0,\ldots, n \}$, and distinguished
subsets $\Ve^i(C) \subset \Ve(C)$ corresponding to irreducible
components on which the $i$-th scaling $\lambda_i$ is finite,
satisfying combinatorial versions of the monotone and balanced
conditions which we leave to the reader to write out.  
This ends the definition.
\end{definition}  

\begin{remark}  {\rm (More explanation of the balanced condition)}
\begin{enumerate} 
\item On any irreducible component $C_i$ of $C$ on which
  $\lambda_j,\lambda_k$ are both non-zero and finite,
  $\lambda_j,\lambda_k$ both have a double pole at the same point and
  so have constant ratio $(\lambda_j | C_i) / ( \lambda_k | C_i)$.
\item The balanced condition is equivalent to the condition that for
  each marking $z_i$, if $C_{i,j}^+$ denotes the unique component
  between $z_0$ and $z_i$ on which $\lambda_j$ is finite, then one of
  the three possibilities holds: $C_{i,j}^+ = C_{i,k}^+$ for all $i$
  and the ratio $(\lambda_j | C_{i,j}^+)/( \lambda_k | C_{i,k}^+)$ is
  independent of $i$; or $C_{i,j}^+$ is closer (in the sense of trees)
  to $z_0$ than $C_{i,k}^+$ for all $i$; or $C_{i,k}^+$ is closer to
  $z_0$ than $C_{i,j}^+$ for all $i$.
\end{enumerate} 
\end{remark}

Let $\ol{M}_{n,s}(\bA)$ denote the moduli space of isomorphism classes
of stable $s$-scaled, $n$-marked curves. 

\begin{remark} {\rm (Boundary divisors of the moduli of muliply-scaled lines)} 
The boundary of $\ol{M}_{n,s}(\bA)$ can be described as follows.
\begin{enumerate}
\item For any subset $I \subset \{ 1,\ldots n \}$ of order at least
  two there is a divisor
$$ \iota_I : D_I \to \ol{M}_{n,s}(\bA) $$
and an isomorphism 
$$ \varphi_I: D_I \to \ol{M}_{|I|+1}
 \times \ol{M}_{n - |I| +
  1,s}(\bA) $$
  corresponding to the formation of a bubble containing the markings
  $z_i, i \in I$ with zero scaling on that bubble, and all scalings
  zero on that bubble.
\item For any unordered partition $I_1 \sqcup \ldots \sqcup I_r $ of
  $\{1,\ldots, n\}$ of order at least two with each $I_j$ non-empty
  and non-empty subset $J \subset \{1,\ldots, s\}$ there is a divisor
$$ \iota_{I_1,\ldots,I_r,J} : D_{I_1,\ldots,I_r, J
  } \to \ol{M}_{n,s}(\bA) $$
with an isomorphism 
$$ \varphi_{I_1,\ldots,I_r,J} : D_{I_1,\ldots,I_r,
  J } \to \ol{M}_{r+1,s - |J|}(\bA) \times \prod_{i=1}^r
\ol{M}_{|I_i|+1,|J|}(\bA) $$
corresponding to the formation of $r$ bubbles containing markings
$I_j, j = 1,\ldots, r$ with the scalings $j \in J$ becoming finite
on those bubbles
and infinite on the component containing $z_0$, or, if 
$J  = \{ 1,\ldots, s \}$, 
$$ \varphi_{I_1,\ldots,I_r,J} : D_{I_1,\ldots,I_r, J } \to
\ol{M}_{r+1} \times \prod_{i=1}^r \ol{M}_{|I_i|+1,s}(\bA) .$$
\end{enumerate} 

\noindent The union of these divisors is the boundary of $M_{n,s}$:
$$ \partial \ol{M}_{n,s}(\bA) = \bigcup_{I \subset \{ 1,\ldots n \}}
D_I \cup \bigcup_{I_1,\ldots,I_r,J} D_{I_1,\ldots,I_r,J} .$$
\end{remark} 

\begin{definition} {\rm (Composition of morphisms of CohFT algebras)} 
Let $U_0,U_1,U_2$ be CohFT algebras.  Given morphisms 
$$\phi_{01}: U_0 \to U_1, \ \ \phi_{12}: U_1 \to U_2, \phi_{02}:U_0
\to U_2$$
we say that $\phi_{02}$ is the {\em composition} of $\phi_{01},
\phi_{12}$ if the map
$$ \phi_{02} \circ \iota_1: U_0^n \times H(\ol{M}_{n,2}(\bA)) \to U_2  $$
given by composing $\phi_{02}$ with the natural restriction map
$H(\ol{M}_{n,2}(\bA)) \to H(\ol{M}_{n,1}(\bA))$ agrees with the map
\begin{multline} (\phi_{12} \circ \phi_{01})^n: U_0^n \times H(\ol{M}_{n,2}(\bA)) \to U_2  \\
 (\alpha_1,\ldots,\alpha_n, \beta) \mapsto \sum_{r \leq s, I_1 \sqcup \ldots \sqcup I_r = \{ 1,\ldots, n \} }
\phi_{12}^s(\phi_{01}^{|I_1|}(\alpha_i,i \in I_1; \cdot ),  \\ \ldots,
\phi_{01}^{|I_r|}(\alpha_i, i \in I_r; \cdot), \phi_{01}^0(1),\ldots, 
\phi_{01}^0(1); \cdot)( \iota^*_{I_1,\ldots,I_r} \beta )/(s-r)! ,\end{multline}
where the dots indicate insertion of the Kunneth components of $
\iota^*_{I_1,\ldots,I_r}(\beta)$ with respect to the Kunneth
decompositions, and is well-defined if it involves only finite sums on
the right-hand-side (modulo $U_{0,a}$ for any $a \in \R$ if all CohFT
algebras are filtered.)  We call the resulting diagram a {\em
  commutative triangle of CohFT}.  Similarly one can define {\em
  commutative simplices} of CohFT algebras of higher dimension.
\end{definition} 

We now define a moduli space of {\em multiply scaled gauged maps} that
``lives above'' $\ol{M}_{n,s}(\bA)$.  Consider a chain of normal
subgroups $G = G_0 \supset G_1 \supset \ldots \supset G_s.$ Since
$G_j$ is normal and compact, $\g$ splits as a sum $\g = \g_j \oplus
\g_j'$, so there exists a subgroup $G_j' \subset G$ so that $G_j
\times G_j' \to G$ is a finite cover.  Let $X$ be a smooth projective
$G_\C$-variety.

\begin{definition} {\rm (Multiply-scaled affine gauged maps)}  An
 {\em $s$-scaled, $n$-marked stable affine gauged map} on the affine
 line $\bA$ with values in $X$ is a $s$-scaled, $n$-marked nodal curve
 $C$ equipped with a morphism $u$ from $C$ to the quotient stack $
 X/G_\C$ such that for each $j = 1,\ldots, s$,
\begin{enumerate}
\item ($G_{j,\C}'$-bundle where $\lambda_j$ is zero) on the
  irreducible components where $\lambda_j$ vanishes, the
  $G_{\C}$-bundle defined by the composition of $u$ with $X/G_\C \to
  B(G_\C)$ is induced from a $G_{j,\C}'$-bundle; 
\item ($G_{j,\C}$-stable point where $\lambda_j$ is infinite) if
  $\lambda_j(z) = \infty$, then $u(z)$ lies in the semistable locus
  for the action of $G_{j,\C}$.
\end{enumerate} 
An $s$-scaled nodal affine gauged map 
is {\em semistable} if each
irreducible component with some non-degenerate scalings has at least
two special points, and each bubble with all degenerate scalings has
at least three special points.  A multiply scaled affine gauged map is
{\em stable} if it has finite automorphism group.
\end{definition} 

Let $\ol{M}_{n,s}^G(\bA,X)$ denote the moduli space of isomorphism
classes of stable $s$-scaled, $n$-marked affine gauged maps on $\C$
with values in $X$. 

\begin{remark} \label{equivalences}
The divisor relations on $\ol{M}_{n,s}(\bA)$ naturally induce divisor
relations on $\ol{M}_{n,s}^G(\bA,X)$.  In particular,
$\ol{M}_{1,2}(\bA)$ is a projective line, and the linear equivalence
between $D_{ \{ 1 \}, \{ 1 \}}$, the divisor where the first scaling
has become infinite, and the subspace $\ol{M}_{1,1}(\bA)$ where the
two scalings have become equal induces an equivalence in homology in
$\ol{M}_{n,2}^G(\bA,X)$ between $\ol{M}_{n,1}^G(\bA,X)$ (embedded as
the subspace where the scalings are equal) and the union of the
pre-images of the divisors $D_{[I_1,\ldots,I_r], \{ 1 \} }$.
\end{remark} 

\begin{remark} {\rm (Equivariant quantum Kirwan morphism)} 
The quantum Kirwan morphism has the following equivariant
generalization.  If the action of $G$ extends to an action of a group
$K$ containing $G$ as a normal subgroup, then the quotient group $K/G$
acts on the moduli space $\ol{M}_{n,1}^G(\bA,X)$ and one obtains a
morphism
$$ \ev \times \ev_\infty: \ol{M}_{n,1}^G(\bA,X)/(K/G)_\C \to (X/K_\C)^{n} \times (X
\qu G)/(K/G)_\C .$$
Pairing with the virtual fundamental class defines a map 
$$ QH_K(X,\Q)^{n} \times H(\ol{M}_{n,1}(\bA),\Q) \to
QH_{G/K}(X \qu G,\Q).
$$
After extending the coefficient ring of $QH_{K/G}(X \qu G)$ from
$\Lambda_X^G$ to $\Lambda_{X \qu G}^{K/G}$ one expects this to
define a morphism of CohFT algebras
\begin{equation} \label{qpartial}
(Q\kappa_{K,G}^n)_{n \ge 0} : QH_K(X) \to QH_{K/G}(X \qu G)
  .\end{equation}
\end{remark} 

Consider the equivariant quantum Kirwan morphisms
$$ (Q\kappa_{K,G}^n)_{n \ge 0} : QH_K(X) \to QH_{K/G}(X \qu G) $$ 
$$ (Q\kappa_{K/G}^n)_{n \ge 0}: QH_{K/G}(X \qu G) \to QH(X \qu
K) $$
defined in \eqref{qpartial}.  The linear equivalence in Remark
\ref{equivalences} leads naturally to the

\begin{conjecture}  {\rm (Quantum reduction in stages)} 
Suppose that $X, K,G$ are as above, and the symplectic quotients by
$K$ and $G$ are locally free.  Then there is a commutative triangle of
CohFT algebras
$$\begin{diagram} \node{QH_K(X)} \arrow{se} \arrow[2]{e} \node{}
\node{QH(X \qu K)} \\ \node{} \node{QH_{K/G}(X \qu G)}\arrow{ne}
\node{} \end{diagram}$$
\end{conjecture} 
\noindent In particular, there is an equality of formal, non-linear maps
$$ Q \kappa_{K/G} \circ Q\kappa_{G,K} = Q\kappa_K .$$
More generally, given a chain $G = G_0 \supset G_1 \supset \ldots G_s$
as above one should obtain a {\em commutative simplex} of
CohFT algebras.  We leave it to the reader to formulate the precise
conjecture.

\section{Local structure of boundary divisors}
\label{local}

In this section and the next we give a precise description of the
group of invariant Cartier divisors on the moduli space of scaled
lines $\ol{M}_{n,1}(\bA)$.  We begin with a review of the local
description of $\ol{M}_{n,1}(\bA)$ given in Ma'u-Woodward
\cite{mau:mult}.

\begin{definition} {\rm (Colored trees)}  A {\em colored tree} $\Gamma$ is a finite connected 
tree consisting of a set of vertices
$$\Ve(\Gamma) = \{ v_1,\ldots, v_m \} $$ 
a set of (finite and semi-infinite) edges
$$\Edge(\Gamma) = \Edge_{< \infty}(\Gamma)
\cup \Edge_\infty(\Gamma) , \quad \Edge_\infty(\Gamma) = \{ e_0,\ldots, e_n \} $$
and a subset of {\em colored vertices}
$$\Ve^+(\Gamma) \subset \Ve(\Gamma)$$ 
such that the following condition is satisfied:
\begin{itemize}
\item[] {\rm (Monotonicity condition)} Any non-self-crossing path in
  $\Gamma$ from the root edge $e_0$ to any other semi-infinite edge
  $e_i, i > 0$ crosses exactly one colored vertex $v \in
  \Ve^+(\Gamma)$.
\end{itemize}
The tree $\Gamma$ is {\em stable} if the valency of any colored
resp. uncolored vertex is at least two resp. three.
\end{definition} 
 
We say that a vertex is ``above'' the colored vertices if it can be
connected to the root edge without crossing a colored vertex.  Let
$\Ve^\infty(\Gamma)$ be the set of vertices above the colored
vertices.  For any $v \in \Ve^\infty(\Gamma)$, let $\Ve^+(v)$ be the
set of colored vertices $v' \in \Ve^+(\Gamma)$ that are below
$\Ve^+(v)$, that is, connected by paths in $\Gamma$ that move away
from the root.

\begin{definition} 
 \label{balanced} {\rm (Balanced Labellings)}
A map $ \varphi: \Edge_{<\infty}(\Gamma) \to \C$ is {\em balanced} if
for all $v \in \Ve^\infty(\Gamma)$ and $v' \in \Ve^+(v)$, the product
$$ \prod_{e \in \gamma(v,v')} \varphi(e) $$
over edges $e$ in the non-self-crossing path $\gamma(v,v')$ from $v$
to $v'$ is independent of the choice of a colored vertex $v'$.
Let $V(\Gamma)$ denote the set of balanced labellings:
\begin{equation} \label{toric}
 V(\Gamma) := \{ \varphi:\Edge_{<\infty}(\Gamma) \to \C \, |  \,
\varphi \ \text{is balanced} \}.\end{equation}
\noindent The subset
$$T(\Gamma) := V(\Gamma) \cap \Map(\Edge_{<\infty}(\Gamma),\C^\times) $$
of points with non-zero labels is the kernel of the homomorphism 
$$ \Map(\Edge_{<\infty}(\Gamma),\C^\times) \to \Map(\Ve^\infty(\Gamma),\C^\times) $$
given by taking the product of labels from the given vertex to the
colored vertex above it, and is therefore an algebraic torus.  
\end{definition}

\begin{example} \label{prev} 
The tree $\Gamma$ in Figure \ref{fourex} is a balanced colored tree
with $n = 4$ and $ \nodes = 3$.  The space of balanced labellings is
$$V(\Gamma) = \lbrace(x_1, ..., x_6) \in \mathbb{C}^6 \hspace{0.1 cm}
| \hspace{0.1 cm} x_1 x_3 = x_1x_4 = x_2x_5 = x_2x_6, x_3 = x_4, x_5 =
x_6 \rbrace$$
and admits an action of the torus 
$$T(\Gamma) = \lbrace (x_1, ..., x_6)
\in V(\Gamma) \hspace{0.1 cm} | \hspace{0.1 cm} x_i \not= 0 \rbrace
\simeq (\mathbb{C}^*)^3 .$$
\end{example}

\begin{proposition}
\label{toricprop}  \cite{mau:mult} 
{\rm (Local structure of the moduli space of scaled lines)} 
There exists an
isomorphism of a Zariski open neighborhood of $M_{n,1,\Gamma}$ in
$M_{n,1,\Gamma} \times V(\Gamma)$ with a Zariski open neighborhood of
$M_{n,1,\Gamma}$ in $\ol{M}_{n,1}(\bA)$.
\end{proposition}    

We comment briefly on the proof.  Given a stable scaled line, one can
remove small disks around the nodes and glue together annuli using a
map $z \mapsto \varphi(e)/z$ to produce a curve with fewer nodes,
where $\varphi(e)$ is the {\em gluing parameter} associated to the
node.  In the case of the genus zero curves, the local coordinates
used to produce the disks are essentially canonical, and the balanced
condition guarantees that the scalings on the resulting curve are
well-defined.

\begin{figure}[h]
\begin{picture}(0,0)%
\includegraphics{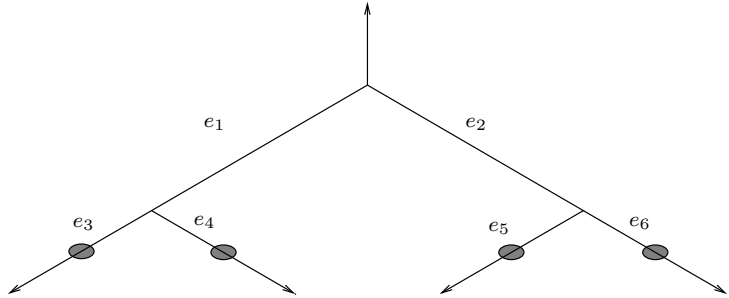}%
\end{picture}%
\setlength{\unitlength}{3947sp}%
\begingroup\makeatletter\ifx\SetFigFont\undefined%
\gdef\SetFigFont#1#2#3#4#5{%
  \reset@font\fontsize{#1}{#2pt}%
  \fontfamily{#3}\fontseries{#4}\fontshape{#5}%
  \selectfont}%
\fi\endgroup%
\begin{picture}(4544,1864)(2389,-1613)
\put(5277,-550){\makebox(0,0)[lb]{\smash{{\SetFigFont{8}{9.6}{\rmdefault}{\mddefault}{\updefault}{\color[rgb]{0,0,0}$e_2$}%
}}}}
\put(2812,-1180){\makebox(0,0)[lb]{\smash{{\SetFigFont{8}{9.6}{\rmdefault}{\mddefault}{\updefault}{\color[rgb]{0,0,0}$e_3$}%
}}}}
\put(3572,-1167){\makebox(0,0)[lb]{\smash{{\SetFigFont{8}{9.6}{\rmdefault}{\mddefault}{\updefault}{\color[rgb]{0,0,0}$e_4$}%
}}}}
\put(5421,-1207){\makebox(0,0)[lb]{\smash{{\SetFigFont{8}{9.6}{\rmdefault}{\mddefault}{\updefault}{\color[rgb]{0,0,0}$e_5$}%
}}}}
\put(6305,-1172){\makebox(0,0)[lb]{\smash{{\SetFigFont{8}{9.6}{\rmdefault}{\mddefault}{\updefault}{\color[rgb]{0,0,0}$e_6$}%
}}}}
\put(3634,-550){\makebox(0,0)[lb]{\smash{{\SetFigFont{8}{9.6}{\rmdefault}{\mddefault}{\updefault}{\color[rgb]{0,0,0}$e_1$}%
}}}}
\end{picture}%
\caption{A colored tree}
\label{fourex}
\end{figure}

Recall that normal affine toric varieties are classified by finitely
generated cones \cite{danilov:toric}, \cite{fu:in}.

\begin{definition}
\label{affinetoric}
 {\rm (Affine toric variety associated to a cone)} Let $V_\Z \subset
 V$ be a lattice, and $C \subset V$ a strictly convex rational cone.
 The {\em affine toric variety corresponding to the cone $C$} is the
 spectrum $V(C)$ of the ring $R(C^\dual)$ corresponding to the
 semigroup $C^\dual \cap V_\Z^\dual$, that is, the ring generated by
 symbols $f_\mu$ for $\mu \in C^\dual \cap V_\Z^\dual$ modulo the
 ideal generated by relations
\begin{equation} \label{toricrelns} 
\sum_i n_i \mu_i = \sum_j m_j \mu_j  \implies 
\prod_i f_{\mu_i}^{n_i} = \prod_j f_{\mu_j}^{m_j}.
\end{equation} 
Any normal affine toric variety is of the form $V(C)$ for some cone
$C$, obtained from $X$ by letting $C^\dual$ be the cone generated by
the weights of the action of $T$ on the coordinate ring and $C$ the
dual cone of $C^\dual$.
\end{definition} 

We wish to show that $V(\Gamma)$ of balanced labellings \eqref{toric}
is the toric variety associated to some cone $C(\Gamma)$.  Note that
the part of $\Gamma$ separated by the colored vertices from the root
of $\Gamma$ trivially affects $V(\Gamma)$ with additional independent
variables. Hence, for the rest of this section, it suffices to assume
that the colored tree $\Gamma$ does not contain any vertex below any
colored vertex.  Let $\eps_1,\ldots, \eps_n$ be a basis of $\t$, and
$\eps_1^\dual, \ldots, \eps_n^\dual$ the dual basis of $\t^\dual$.
Define a labelling
$$ w: \Edge_{< \infty}(\Gamma) \to \t^\dual $$
recursively as follows.  

\begin{definition} {\rm (Principal subtree and branch)}
We say that a subtree $\Gamma' \subset \Gamma$ is a {\em principal
  subtree} if it is a component of the tree obtained by removing the
vertex adjacent to the root edge.  The edge adjacent to the root edge
of $\Gamma'$ is called a {\em principal branch} of $\Gamma$.
\end{definition} 

Let $\Gamma_1, ..., \Gamma_p$
be the principal subtrees of $\Gamma$ and $d_1,\ldots, d_p$ the
principal branches.  

\begin{example}  For the example in Figure \ref{fourex}, there are
two principal subtrees, with principal branches $e_1,e_2$.
\end{example} 

\begin{definition} {\rm (Sum of Labels)}
Given a labelling denote by $s(\Gamma,w)$ the sum of the labels of the
edges of a non-self-crossing path from the principal vertex to a
colored vertex; a priori this depends on the choice of path but each
labelling we construct will have the property that $s(\Gamma,w)$ is
independent of the choice of path.
\end{definition} 

\begin{definition}  {\rm (Labelling of edges of a colored
tree by weights)} Let $\Gamma'$ be a subtree of $\Gamma$.
\begin{enumerate} 
\item[(Case 1)] $\Gamma'$ is a tree with one non-colored vertex
  $v_i$.  Label the edges below the vertex $v_i$ by $\eps_i^\dual$, that
  is, define $w(e) = \eps_i^\dual$ for every edge $e$ of $\Gamma'$.
\item[(Case 2)] $\Gamma'$ has $g >1 $ non-colored
  vertices. By induction, assume that we have equipped the edges of
  the principal subtrees $\Gamma'_1, ..., \Gamma'_p$ of $\Gamma'$ with
  labellings $w_i$.  We have thus labelled all the edges of $\Gamma$
  except for the principal branches; we denote $s_i :=
  s(\Gamma'_i,w_i)$.  Define
\begin{equation} \label{s} s = s(\Gamma') = \eps_g^\dual + s_1+ ... + s_p = \eps_g^\dual + ... + \eps_1^\dual. \end{equation}
Label the principal branch $d_i \in \Edge_{< \infty}(\Gamma')$ with
\begin{equation} \label{wdi}
w(d_i) = s - s_i = \eps_g^\dual + \sum \limits_{j \neq i} s_j.
\end{equation}
\end{enumerate}
By induction all the edges $e$ of $\Gamma$ become labelled by weights
$w(e)$.
\end{definition} 

\begin{figure}[h]
\begin{picture}(0,0)%
\includegraphics{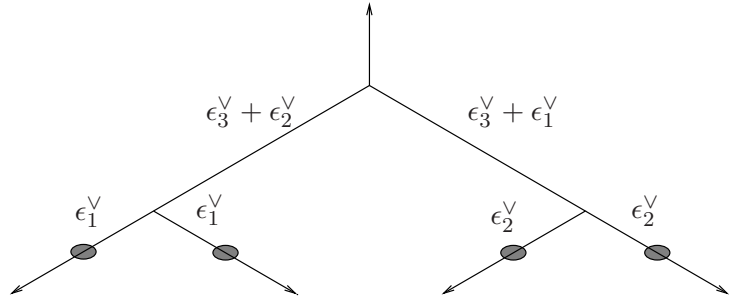}%
\end{picture}%
\setlength{\unitlength}{3947sp}%
\begingroup\makeatletter\ifx\SetFigFont\undefined%
\gdef\SetFigFont#1#2#3#4#5{%
  \reset@font\fontsize{#1}{#2pt}%
  \fontfamily{#3}\fontseries{#4}\fontshape{#5}%
  \selectfont}%
\fi\endgroup%
\begin{picture}(4544,1864)(2389,-1613)
\put(3572,-1167){\makebox(0,0)[lb]{{$\eps_1^\dual$}%
}}
\put(5421,-1207){\makebox(0,0)[lb]{{$\eps_2^\dual$}%
}}
\put(6305,-1172){\makebox(0,0)[lb]{{$\eps_2^\dual$}%
}}
\put(3634,-550){\makebox(0,0)[lb]{{$\eps_3^\dual + \eps_2^\dual$}%
}}
\put(2812,-1180){\makebox(0,0)[lb]{{$\eps_1^\dual$}%
}}
\put(5277,-550){\makebox(0,0)[lb]{{$\eps_3^\dual + \eps_1^\dual$}%
}}
\end{picture}%
\caption{An example of a labelling}
\end{figure}

\begin{example}
Figure 2 illustrates the labels of the edges of $\Gamma$ from Example \ref{prev}.
If we denote the left and right principal branches by $d_1$ and $d_2$ respectively, 
then $s_1 = \eps_1^\dual, s_2 = \eps_2^\dual, s = \eps_3^\dual + \eps_2^\dual + \eps_1^\dual.$ 
\end{example} 

\begin{lemma}  $s = s(\Gamma,w)$ is the sum of the labels of the edges of a
  non-self-crossing path from the principal vertex $v_g$ to a colored
  vertex and $s$ is independent of the path chosen. 
\end{lemma} 

\begin{proof}  By \eqref{wdi} the sum over a path through $\Gamma_i$
is $s_i + w(d_i) = s$, independent of $i$. 
\end{proof}  

Let $C(\Gamma)^\dual$ be the convex cone generated by the labels
above,
\begin{eqnarray*}
C(\Gamma)^\dual &=&  \on{hull}_{\Q_{\geq 0}} \ \{ w(e) \hspace{0.1 cm} | \hspace{0.1 cm}
e \in \Edge_{<\infty}(\Gamma) \} \\
&=& \on{hull}_{\Q_{\geq 0}} \ \bigcup \limits_{j = 1}^p C(\Gamma_j)^\dual
\cup \lbrace w(e_i) \hspace{0.1 cm} | \hspace{0.1 cm} 1 \leq i \leq p
\rbrace
\end{eqnarray*}
and let $C(\Gamma)$ denote the dual cone of $C(\Gamma)^\dual$. 

\begin{theorem} \label{spec} 
{\rm (Explicit description of the cone associated to balanced
  labellings)} The variety $V(\Gamma)$ is the toric variety associated
to $C(\Gamma)$ in the sense of Definition \ref{affinetoric}, in
particular, $V(\Gamma)$ is normal.
\end{theorem}  
\noindent The proof will be given after the following lemma. 

\begin{definition} {\rm (Equivalence of sets of edges)} 
Suppose $E', E''$ are two disjoint subsets of $\Edge(\Gamma)$.  We
write $E' \sim E''$ if there exists a vertex and two non-self-crossing
paths $\gamma_1$ and $\gamma_2$ from that vertex to some two colored
vertices so that $E'$ and $E''$ respectively contain exactly the edges
of the paths $\gamma_1$ and $\gamma_2$.
\end{definition}

\begin{example}  The set $E' = \{e_1,e_3 \}$ is equivalent
to $E'' = \{ e_2,e_6 \}$ in Example \ref{prev}.
\end{example}

\begin{lemma}\label{disj} 
 Suppose $E'$ and $E''$ are two disjoint multisets of elements of
 $\Edge(\Gamma)$. Then
\begin{equation} \label{equality} \sum \limits_{e' \in E'} w(e') = \sum \limits_{e'' \in E''} w(e'') \hspace{0.6 cm} \end{equation}
if and only if $E'$ and $E''$ can be partitioned into disjoint unions of
$\lbrace E'_1, ..., E'_r \rbrace$ and $\lbrace E''_1, ..., E''_r \rbrace$
where $E'_l \sim E''_l$ for $1 \leq l \leq r$.
\end{lemma}

\begin{example} In Example \ref{prev}, let $E' = \{ 3e_1, 2e_3, e_4, e_5 \}$ and $E'' = \{3e_2, 4e_6 \}$ which satisfy \eqref{equality}.   We can write $E'_1 = \{ e_1, e_3\} \sim E''_1 = \{e_2, e_6\}, E'_2 = \{ e_1, e_3\} \sim E''_2 = \{e_2, e_6\}, E'_3 = \{ e_1, e_4\} \sim E''_3 = \{e_2, e_6\}, E'_4 = \{ e_5 \} \sim E''_4 = \{e_6\} $.
\end{example} 

\begin{proof}[Proof of Lemma \ref{disj}] One direction of the implication, that the equality 
\eqref{equality} holds if $E'$ and $E''$ can be partitioned, is
immediate from the definitions.  We only need to show the other
direction.  As before, it suffices to consider the case that there are
no non-colored vertices below the colored vertices.  When the number
of non-colored vertices is $1$, \ccomment{removed space} the statement of the lemma is
trivial.  Assume the proposition holds for any tree with number of
vertices less than $g$.  Consider a tree $\Gamma$ with $\nodes$
non-colored vertices.  Denote by $n_1, ..., n_p$ and $m_1, ..., m_p$
the multiplicities of the principal branches $e_{1}, ..., e_p$ in $E'$
and $E''$.  Since $E' \cap E'' = \emptyset$, we have $n_im_i= 0$.
Equation \eqref{equality} and the fact that $\eps_g^\dual$ appears only on
the edges adjacent to the root edge implies
\begin{equation} \label{equality1} \sum \limits_{i = 1}^p n_i 
= \sum \limits_{i = 1}^p m_i .\end{equation}
Similarly, the fact that the labels from each principal branch
are independent implies that 
\begin{equation*} 
-n_is_i + \sum 
\limits_{e' \in E' \cap \Edge_{<\infty}(\Gamma_{i})} w(e') = -m_is_{i} + \sum
\limits_{e'' \in E'' \cap \Edge_{<\infty}(\Gamma_{i})} w(e'')  \end{equation*} 
for every $1 \leq i \leq p$. For a fixed $i$, without loss of generality, we can assume
$m_i = 0$. Then
$$\sum \limits_{e' \in E' \cap \Gamma_{i}} w(e') = n_is_i +  
\sum \limits_{e'' \in E'' \cap \Gamma_{i}} w(e'').$$
Noting that $ s_i$ is the sum of labels over a non-self-crossing path
from $v_i$ to a colored vertex, we may replace $E'^i = E' \cap
\Edge_{<\infty}(\Gamma_{i})$ with an equivalent set which contains
$n_i$ copies $E'^i_j, j = 1,\ldots, n_i$ of the edges in such a path.
The complement of $E'^i_j, j = 1,\ldots, n_i$ in $E'^i$ has the same
sum of labels as $E''^i = E'' \cap \Edge_{< \infty}(\Gamma_i)$, so by
the inductive hypothesis there exists a partition of $E' \cap
\Edge_{<\infty}(\Gamma_{i})$ and $E'' \cap
\Edge_{<\infty}(\Gamma_{i})$ into $\lbrace E'^i_1, ..., E'^i_{r_i'}
\rbrace$ and $\lbrace E''^i_{1}, ..., E''^i_{r_i''} \rbrace$ such that
for $1 \leq j \leq n_i$, $E'^i_{j}$ are equal and for $n_i + 1 \leq j
\leq r_i$,
\begin{equation} \label{sim} E'^i_{j} \sim  E''^i_{j} .\end{equation} 
Since $E'$ contains $n_i$ principal branches $d_i$, we can add one
edge $d_i$ in each $E'^i_j$ for every $1 \leq j \leq n_i$. Hence,
after the modification, each set $E'^i_j$ contains exactly all the
edges of a path from the root of $\Gamma$ to a colored vertex in
$\Gamma_i$ .  Applying the same process for each $1 \leq i \leq p$, by
the first equality in \eqref{equality1} and by \eqref{sim}, we can
partition $E'$ and $E''$ into $\{ E'_1, ..., E'_r\}$ and $\{E''_1,
..., E''_r \}$ such that $E'_i \sim E''_i$ for every $1 \leq i \leq
r$.
\end{proof}    

\begin{proof}[Proof of Theorem \ref{spec}]
We must show that the balanced relations for $V(\Gamma)$ in Definition
\ref{balanced} are exactly those in the definition of the affine toric
variety associated to $C(\Gamma)$ in \eqref{toricrelns}.  So suppose
that $E' = \{ n_1 e_1, ..., n_N e_N \}$ and $E'' = \{m_1 e_1, ..., m_N
e_N \}$ are such that $\sum n_i w(e_i) = \sum m_j w(e_j)$, and so
define a relation as in \eqref{toricrelns}.  Lemma \ref{disj} yields
that $E'$ and $E''$ can be partitioned into $E'_1, ..., E'_r, E''_1,
..., E''_r$ so that $E'_i \sim E''_i$ for $1 \leq i \leq r$.  But
these are exactly the balanced relations in \ref{balanced}.
\end{proof} 

It follows from the theorem that the cone $C(\Gamma)$ corresponding to
the toric variety $V(\Gamma)$ is the cone dual to the $\Q_{\geq
  0}$-span of $C(\Gamma)^\dual$.  Next we find a minimal set
$G(\Gamma)$ of generators of $C(\Gamma)$ by an inductive argument on
the number of non-colored vertices $\nodes$ of $\Gamma$.
\begin{definition} {\rm (Generators of the cone associated to balanced labellings)} 
Define $G(\Gamma)$ inductively as follows for subtrees $\Gamma'
\subset \Gamma$:
\begin{enumerate} 
\item If $\nodes(\Gamma') = 1$ with vertex $v_i$, then $G(\Gamma') =
  {\eps_i}$.
\item If $\nodes > 1$, then
$G(\Gamma') = \lbrace \eps_{\nodes} + n_1(v_1 - \eps_{\nodes}) + ... + n_p(v_p -
  \eps_{\nodes}) 
| \hspace{0.1 cm} v_i \in
  G(\Gamma_i'), n_i \in \lbrace 0, 1 \rbrace \rbrace. $
\end{enumerate} 
\end{definition} 

Note that the elements in $G(\Gamma)$ are in the lattice $\Z^g$ spanned by the vectors $\eps_1, ..., \eps_g$. 

\begin{theorem}\label{generators} $G(\Gamma)$ is a minimal set of generators
of $C(\Gamma)$.
\end{theorem} 

\begin{example} 
The tree $\Gamma$ in Figure 1 can be split into two principal subtrees
$\Gamma_1$ and $\Gamma_2$, as illustrated in Figure 2. Since
$G(\Gamma_1) = \lbrace \eps_1 \rbrace$ and $G(\Gamma_2) = \lbrace \eps_2
\rbrace$, we obtain $G(\Gamma) = \lbrace \eps_1, \eps_2, \eps_3, \eps_1 + \eps_2 -
\eps_3 \rbrace$. The cone generated by $\lbrace \eps_1, \eps_2, \eps_3, \eps_1 + \eps_2
- \eps_3 \rbrace$ is the cone $C(\Gamma)$ corresponding to the toric
variety $V(\Gamma)$.
\end{example}

Denote by $\tilde{C}(\Gamma)$ the $\nodes$-dimensional cone spanned by
the vectors in $G(\Gamma)$.  To prove Theorem \ref{generators} we must
show that $C(\Gamma) = \tilde{C}(\Gamma)$.  

\begin{lemma} \label{pair} For $s$
as in \eqref{s}, for every $v \in G(\Gamma)$, $ \langle s, v \rangle =
1.$
\end{lemma} 

\begin{proof} 
This follows by induction on the number of vertices from the
observation that
$\langle s, v \rangle = 1 + \sum \limits_{i=1}^p n_i (\langle
s_i, v_i \rangle - 1) .$
\end{proof} 

\begin{proof}[Proof of Theorem \ref{generators}]
We show $\tilde{C}(\Gamma) \subset C(\Gamma)$ by induction on the
number $g$ of non-colored vertices.  The case $g = 1$ is obvious.
Suppose the claim is true for all colored trees with less than $g$
vertices.  Let $v \in G(\Gamma)$ with coefficients $n_i, i = 1,\ldots,
p$ and $w \in C(\Gamma)^\dual$.  If $v \in G(\Gamma_i)$ \ccomment{v
  not w} then $\langle w, v \rangle = n_i \langle w, v_i \rangle$ by
the inductive hypothesis.  Otherwise, since $C(\Gamma)^\dual$ is
spanned by $w(e), e \in E_{< \infty}(\Gamma)$, we may assume that $w =
w(e_i)$ for some $1 \leq i \leq p.$  By Lemma \ref{pair},
\begin{eqnarray*} 
\langle w, v \rangle &=& \langle s - s_i, v \rangle \\
&=& \langle s, v_i \rangle - n_i 
\langle s_i, v_i \rangle \\
&=& 1 - n_i \end{eqnarray*} 
Hence $\langle w, v \rangle \geq 0$.  Since this holds for all $v,w$,
we have $\tilde{C}(\Gamma) \subset C(\Gamma)$. \ccomment{changed
  wording, fixed mistakes}

Conversely, given $v \in C(\Gamma)$, we claim that $v$ is a
non-negative linear combination of elements in $G(\Gamma)$. For
$\nodes = 1$, the claim is trivial.  Assume the claim is true for all
trees with less than $g$ non-colored \ccomment{added non-colored}
vertices.  Let $\Gamma$ be a tree with $g$ vertices and $v \in
C(\Gamma)$.  In particular, $v$ pairs non-negatively with the weights
$w(e), e \in \Edge_{< \infty}(\Gamma_i)$ so by the inductive
hypothesis we can write $v$ as a sum
$$v = -c_{\nodes} \eps_{\nodes} + \sum_{i=1}^p 
\sum_{v
  \in G(\Gamma_i)} \lambda^{(i)}_v v$$
where $\lambda^{(i)}_v \geq 0$.  If $c_{\nodes} \leq 0$ then the claim
follows since $\eps_{\nodes} \in G(\Gamma)$.  If $c_{\nodes} > 0$, let
$$\lambda_i := \sum \limits_{v \in G(\Gamma_i)} \lambda^{(i)}_v.$$ 
We remark by Lemma \ref{pair},
\begin{equation} \label{nonneg} 
0 \leq \langle w(e_i), v \rangle = -c_{\nodes} + \sum \limits_{j \not=
  i} {\lambda_j} .\end{equation}
To write $v$ as a non-negative linear combination of elements in
$G(\Gamma)$, we proceed as follows.  Without loss of generality,
suppose that $\lambda_p$ is the minimum of $\{ \lambda_j \neq 0 \}$,
that is, the smallest positive $\lambda_j, j = 1,\ldots, p$.  Split
each sum
$$
\sum_{v
  \in G(\Gamma_i)} \lambda^{(i)}_v v
 = \sum \limits_{v \in G(\Gamma_i)} \gamma^{(i)}_v v+ \sum
\limits_{v \in G(\Gamma_i)} \delta^{(i)}_v v $$ 
where $\gamma^{(i)}_v, \delta^{(i)}_v \geq 0, \gamma^{(i)}_v +
\delta^{(i)}_v = \lambda^{(i)}_v$ and $\sum \limits_{v \in
  G(\Gamma_i)} \delta^{(i)}_v = \lambda_p.$
We can write $v$ as the sum of 
\begin{equation} \label{first} 
-(p-1)\lambda_p \eps_{\nodes} + \sum \limits_{i=1}^{p-1}\big(\sum
\limits_{v \in G(\Gamma_i)} \delta^{(i)}_v v \big) + \sum \limits_{v
  \in G(\Gamma_p)} \lambda^{(p)}_v v \end{equation} 
and
\begin{equation} \label{second} 
-c_{\nodes}'\eps_{\nodes} + \sum \limits_{i = 1}^{p-1} \sum \limits_{v \in G(\Gamma_i)} \gamma^{(i)}_v v
\end{equation} 
where
\begin{equation} \label{third}
 \quad -c_{\nodes}' = -c_{\nodes} + (p-1)\lambda_p.
\end{equation} 
Since
$$\sum \limits_{v \in G(\Gamma_i)} \delta^{(i)}_v = \sum \limits_{v
  \in G(\Gamma_p)} \lambda_{v}^{(p)} = \lambda_p,$$ 
the expression \eqref{first} is a non-negative linear combination of
elements of $G(\Gamma)$.  If $-c_{\nodes}' \geq 0$, the expression
\eqref{second} is already a nonnegative linear combination of elements
of $G(\Gamma)$ and we are done.  Otherwise, consider the smaller tree
$\Gamma'$ obtained from $\Gamma$ by removing $\Gamma_p$ and observe
that \eqref{second} lies in $C(\Gamma')$. Indeed, by definition, we
know $\gamma_v^{(i) }\geq 0$ and therefore it is sufficient to check
that
$$-c_{\nodes}'+ \sum \limits_{j \not= i, j = 1}^{p-1} \gamma_j \geq 0.$$
However, by definition and the equation \eqref{nonneg} we have
\begin{eqnarray*}
-c_{\nodes}'+ \sum \limits_{j \not= i, j = 1}^{p-1} \gamma_j &=& -c_{\nodes} +
(p-1)\lambda_p + \sum \limits_{j \not= i, j =1}^{p-1} (\lambda_j -
\lambda_p) \\
&=& -c_{\nodes} + \sum \limits_{j=1, j \not= i}^{p} \lambda_i \geq 0.\end{eqnarray*}
Thus, the expression \eqref{second} is in $C(\Gamma')$.  By the
inductive hypothesis, \eqref{second} is a nonnegative linear
combination of elements of $G(\Gamma')$.  Hence $v$ is a nonnegative
linear combination of elements in $G(\Gamma)$. Thus $C(\Gamma) \subset
\tilde{C}(\Gamma)$ and therefore, $C(\Gamma) = \tilde{C}(\Gamma).$
\ccomment{changed wording, added eqn nos} 

We argue by induction that $G(\Gamma)$ is a minimal set of generators of
$\tilde{C}(\Gamma)$ and $\tilde{C}(\Gamma)$ is nondegenerate, i.e no 
positive linear combinations of vectors in the cone are 0. It is easy to check
the claim when $g(\Gamma) = 1$. Given a tree $\Gamma$ with $g(\Gamma)$
non-colored vertices and $G(\Gamma_i)$ the constructed minimal set of 
generators for each nondegenerate cone $\tilde{C}(\Gamma_i)$,  suppose $v \in
G(\Gamma)$ is a non-negative linear combination of other elements in
$G(\Gamma)$.  By the induction hypothesis,  the projection of $v$ onto the space 
spanned by $G(\Gamma_i)$ is 0 for each  $i$ and thus by the nondegeneracy induction
hypothesis, it follows that $v = 0$. Now, suppose that a positive linear combination
of some elements in $G(\Gamma)$ is 0. In particular, its projections onto the space 
spanned by $G(\Gamma_i)$ is 0 for each $i$ and hence by the nondegeneracy induction
hypothesis, it follows that all the elements in the combination are 0. Therefore, $G(\Gamma)$ 
is a minimal set of generators of $\tilde{C}(\Gamma)$ and $\tilde{C}(\Gamma)$ is nondegenerate,
concluding the theorem.
 \end{proof}

By the description of the cone, the dimension of $V(\Gamma)$ equals
the number of non-colored vertices $\nodes$ above the colored
vertices plus the number of edges below the colored vertices.  On the
other hand, by the balanced condition in \ref{balanced},
$$\dim(V(\Gamma)) = \dim(T(\Gamma)) =|\Edge_{< \infty}(\Gamma)| -|\Ve^+(\Gamma)| + 1 .$$
The two formulas are easily seen to be equivalent, by considering the
map from vertices to edges given by taking the adjacent edge in the
direction of the root edge.  We also have a formula for the number of
rays in $C(\Gamma)$, which follows immediately from Theorem
\ref{generators}:
  
\begin{corollary}  \label{one} 
If the number of 1 dimensional faces of $C(\Gamma_i)$ is $r_i$ for $1
\leq i \leq p$, then the number of 1 dimensional faces of $C(\Gamma)$
is $r = (r_1 + 1)...(r_p + 1)$.
\end{corollary}

Next we describe the Weil and Cartier divisors in the local toric
charts.  Recall the description of invariant Weil divisors of an
affine toric variety $V(C)$ with cone $C$, see \cite{fu:in} or in the
more general setting of spherical varieties, \cite{br:gr}:

\begin{proposition} 
\label{weilone}  
\begin{enumerate} 
\item {\rm (Classification of invariant Weil divisors)} There is a
  bijection between invariant prime Weil divisors of $V(C)$ and the
  one dimensional faces of $C$
\item {\rm (Classification of invariant Cartier divisors)} There is a
  bijection between invariant Cartier divisors on $V(C)$ and linear
  functions on $C$ taking integer values on the intersection $C \cap
  V_\Z$.
\end{enumerate} 
\end{proposition} 

We sketch the construction of the bijections.  For the classification
of invariant Weil divisors, any one-dimensional face $C_1$ of $C$
corresponds to a codimension-one face $C_1^\dual$ of $C^\dual$.  The
projection of semigroup rings $R(C^\dual) \to R(C_1^\dual)$ defines an
inclusion of the corresponding affine toric varieties $V(C_1) \to
V(C)$.  \ccomment{changed language} For the classification of Cartier
divisors, recall that a Weil divisor is Cartier if it is the zero set
of a section of a line bundle.  On a normal affine toric variety, any
line bundle is trivial and any invariant Cartier divisor is defined by
a function that is semi-invariant under the torus action.  Such
functions correspond to lattice points $\lambda \in V_\Z^\dual$, where
the corresponding function is regular if $\lambda \in C^\dual \subset
V_\Z^\dual$.  If $v \in C$ is any vector generating an extremal ray,
then the order of vanishing of $\lambda$ on the divisor $D(v) \subset
V(C)$ corresponding to $v$ is $\lambda(v)$.  Thus one sees that a
combination $\sum n_v D(v)$ of invariant Weil divisors is Cartier, iff
there is an element $\lambda \in C^\dual$ such that $\lambda(v) = n_v$
for such $v \in C$.  More generally, for a not-necessarily affine
toric variety, an invariant Weil divisor is Cartier if there exists a
piecewise linear function on the fan whose values on the rays are the
multiplicities of the invariant prime Weil divisors.

We now specialize to the case of the toric variety $V(\Gamma)$
associated to the cone $C(\Gamma)$ with generators $G(\Gamma)$
identified in the previous section.  We identify the invariant prime
Weil divisors of $V(\Gamma)$ as follows.
\begin{definition} [Minimally complete edge sets]  \label{mc}
A subset $E \subset \Edge_{<\infty}(\Gamma)$ is {\em minimally complete} if and
only if each non-self crossing path from $v_g$ to a colored vertex
contains exactly one edge in $E$.
\end{definition}

Denote by $\mE_{mc}(\Gamma)$ the set of minimally complete subsets $E
\subset \Edge_{<\infty}(\Gamma)$.

\begin{example}
The minimally complete subsets of $\Edge(\Gamma)$, where $\Gamma$ 
is the tree in Figure 1, are $\{ e_1, e_2 \}$, $\{ e_1, e_5, e_6 \}$, 
$\{ e_2, e_3, e_4 \}$, $\{ e_3, e_4, e_5, e_6 \}$.
\end{example}

\ccomment{changed prop}
 \begin{proposition}   
\label{mcprop}
%
If the number of minimally complete subsets of
  $\Edge_{<\infty}(\Gamma_i)$ is $r_i$, then the number of minimally
  complete sets of $\Edge_{< \infty}(\Gamma)$ is $r = (r_1 + 1)...(r_p
  + 1).$
  \end{proposition} 

\begin{proof} 
Let $d_1,\ldots, d_p$ denote the edges adjacent to the root edge.
Each minimally complete set either contains $d_i$, or induces a
minimally complete set in the principal branch $\Gamma_i$, for each $i
= 1,\ldots, p$.  The claim follows.
\end{proof}

From Corollary \ref{one} and Proposition \ref{mcprop}, we obtain
\begin{corollary} \label{oneC}
The number of 1 dimensional faces of $C(\Gamma)$ equals the number of 
minimally complete subsets of $\Edge_{<\infty}(\Gamma)$. \end{corollary}

We can now describe the set of invariant Weil divisors of $V(\Gamma)$
as follows.

\begin{proposition} There is a bijection between 
the set of invariant prime Weil divisors and the set
$\mE_{mc}(\Gamma)$. More explicitly, each prime invariant Weil divisor
has the form
$$D_E := \lbrace (x_1, ..., x_N) \in V(\Gamma) \hspace{0.1 cm}
| \hspace{0.1 cm} x_i = 0 \hspace{0.1 cm} \forall x_i \in E \rbrace$$
for some minimally complete subset $E$.
\end{proposition}
\begin{proof}
Given a minimally complete edge set $E \subset E_{< \infty}(\Gamma)$,
for each principal subtree $\Gamma_i$ of $\Gamma$, either $x_{d_i} =
0$ or $D_E$ induces a minimally complete subset $E_i \in
\D(\Gamma_i)$.  By induction on the number of non-colored vertices,
the dimension of $D_E$ is $\nodes(\Gamma_1) + ... + \nodes(\Gamma_p) =
g - 1.$ Thus $D_E$ is a subvariety of $V(\Gamma)$ of codimension 1.
Since $V(\Gamma)$ is the closure of $T(\Gamma)$, the subvariety $D_E$
is the closure of the orbit $D_E \cap T(\Gamma)$ and so a prime Weil
divisor.  From Proposition \ref{weilone} and Proposition \ref{oneC},
the number of prime Weil divisors equals the number of $1$ dimensional
faces of $C(\Gamma)$ which equals the number of minimally complete
subset of $\Edge_{<\infty}(\Gamma)$. Therefore the invariant prime
Weil divisors of $V(\Gamma)$ are exactly all $D_E$, where $E \subset
\Edge(\Gamma)$ is minimally complete.
\end{proof}

\begin{example}
The invariant prime Weil divisors of $V(\Gamma)$ from Example
\ref{prev} are $ D_{ \lbrace 1, 2 \rbrace}, D_{ \lbrace 1, 5, 6
  \rbrace}$, $D_{\lbrace 2, 3, 4 \rbrace}, D_{\lbrace 3, 4, 5, 6
  \rbrace},$ corresponding to the partitions 
$$ \{ \{ 1 \}, \{ 2 \} ,
  \{ 3 \}, \{ 4 \} \}, \{ \{ 1 , 2 \} , \{ 3 \}, \{ 4 \} \}, \{ \{ 1
    \}, \{ 2 \} , \{ 3 , 4 \} \} \{ \{ 1, 2 , 3 , 4 \} \}$$ 
of the labels of the markings $\{1 ,2,3,4 \}$.
\end{example}

We now describe inductively the correspondence between the rays of
$C(\Gamma)$ and elements in $\mE_{mc}(\Gamma)$.  Let $D = D_E$ be a
Weil divisor as above.  Unless $e_{d_i} \in E $, the principal subtree
$\Gamma_i$ has an induced Weil divisor $D_{E_i} \subset V(\Gamma_i)$.
Suppose that $1$ dimensional face of $C(\Gamma_i)$
corresponding to the Weil divisor $D_{E_i}$ is $v_i \in G(\Gamma_i)
\subset G(\Gamma)$. Let
$$I(E) =
\lbrace i \hspace{0.1 cm} | \hspace{0.1 cm} e_{d_i} \notin E, 1 \leq i
\leq p \rbrace.$$

\begin{proposition} Let
$E \in \mE_{mc}(\Gamma)$.  The one-dimensional face of $C(\Gamma)$
  corresponding to the Weyl divisor $D_E \subset V(\Gamma)$ is
  generated by
\begin{equation} v_{E} = \eps_{\nodes} + \sum \limits_{i \in I(E)}(v_i - \eps_{\nodes}).
\end{equation}
\end{proposition}

\begin{proof}
We must show that $v_E$ is non-zero exactly on the
  weights $w(e)$ for $e \in E$.  This is automatically true by the
  inductive hypothesis for the edges except for the principal
  branches, that is, if $e \in \Edge_{< \infty}(\Gamma_i)$ then $\lan
  v_E, w(e) \ran =  \lan v_i, w(e) \ran \neq 0$ iff $e \in E_i$.  For the
  principal branches, the claim follows from \ref{pair}.
\end{proof}

Next we identify the invariant Cartier divisors in $V(\Gamma)$.  For
each vertex $v_k$ of $\Gamma$, denote by $\Gamma_k$ the subtree below
$v_k$ in $\Gamma$ and which contains the edge right above $v_k$ as its
distinguished root.  That is, $\Gamma_k$ is the connected component of
$\Gamma - \{ v_k \}$ not containing the root edge.  We define
$$\D_k = \lbrace D_E  \  | \  E \in \mE_{mc}(\Gamma), \ 
E \cap \Edge_{< \infty}( \Gamma_k ) \neq \emptyset \rbrace,  \quad 
D_k = \sum_{D \in \D_k} D.$$
Hence if $D_E \in \D_k$, $E$ does not contain edges of $\Gamma$ 
which are above $v_k$. 
\begin{proposition} \label{gen} The group of invariant 
Cartier divisors is generated by $D_1, ..., D_g$.
\end{proposition}

\begin{example}
The group of invariant Cartier divisors of $V(\Gamma)$, where $\Gamma$
is the tree in Figure 1, is generated by $$\{ D_{\{ 1,2 \}} + D_{\{ 1,
  5, 6 \}} + D_{\{ 2, 3, 4 \}} + D_{\{ 3, 4, 5, 6 \}}, D_{\{ 2, 3, 4
  \}} + D_{\{ 3, 4, 5, 6 \}}, D_{\{ 1, 5, 6 \}} + D_{\{ 3, 4, 5, 6 \}}
\}.$$
Thus $n_{\{ 1,2 \} }D_{\{ 1,2 \}} + 
n_{\{ 3,4,5,6 \}}D_{\{ 3,4,5,6 \}} + 
n_{\{ 1,5,6 \}}D_{\{ 1,5,6 \}} + 
n_{\{ 2,3,4 \}}D_{\{ 2,3,4 \}}$ 
is a Cartier divisor if and only  if 
$n_{\{ 1,2 \} } +  n_{\{ 3,4,5,6 \}} = 
n_{\{ 1,5,6 \}} + n_{\{ 2,3,4 \}}.$ 
\end{example}

\begin{proof}[Proof of Proposition \ref{gen}] We first check that $D_k$ is a Cartier divisor. 
Recall the notation in \eqref{s}, $s_k = s(\Gamma_k) \in \t_\Z^\dual$
for each vertex $v_k$. Note that $s_k$ satisfies $\langle s_k, v_E
\rangle = 1$ if $E \in \D_k$ and $ \langle s_k, v_E \rangle = 0$
otherwise.  Indeed, for each $E \in \D_k$, $D_E$ induces another
Cartier divisor $D_{E_k}$ and by \ref{pair}, we have $\langle s_k,
v_{E_k} \rangle = 1$. This implies $\langle s_k, v_E \rangle = 1$. On
the other hand, if $E \notin \D_k$, then $E$ does not contain any
edges in $\Gamma_k$.  Thus, $\langle s_k, v_E \rangle = 0.$ Hence
$D_k$ is a Cartier divisor.  

Next we check that $D_k, k = 1,\ldots, g$ generate the group of
invariant Cartier divisors.  \ccomment{added invariant} Note that $s_k = {\eps_k}^\dual \ \on{mod}
\ \eps_1^\dual, \ldots, \eps_{k-1}^\dual$.  It follows by an inductive
argument that $s_1,\ldots, s_g$ generate $\t_\Z^\dual$ so that $D_1, ...,
D_g$ generate the group of Cartier divisors of $V(\Gamma)$.
\end{proof}

We have the following description of the group of 
 invariant Cartier divisors of
$V(\Gamma)$.  Let $w$
\ccomment{changed to lower-case, since all other similar notations are lower-case
(p,n etc.)}  be the number of prime Weil invariant divisors
of $V(\Gamma)$.

\begin{theorem} \label{char} 
 $\sum \limits_{D} n_{D}D$ is a Cartier divisor of $V(\Gamma)$ if and
  only if
$\sum \limits_{D} m_{D}n_{D} = 0$
for every $(m_{D})_D \in \mathbb{Z}^w$ that satisfies $\sum \limits_{E:e
  \in E} m_{D_E} = 0$ for every edge $e \in \Edge_{< \infty}(\Gamma)$. 
\end{theorem}

\begin{proof} The group of 
Cartier boundary divisors of $V(\Gamma)$ forms a sublattice of the
group of Weil boundary divisors $\mathbb{Z}^w$, isomorphic to the
weight lattice $\t_\Z^\dual \cong \Z^g$.  Suppose $(m_D)_{D} \in
{\Z}^w$ satisfies the condition in the Theorem,
$$\sum \limits_{E: e \in E} m_{D_E} = 0, \quad \forall e
\in \Edge_{<\infty}(\Gamma) .$$  
Consider a non-self-crossing path $\gamma$ from a vertex $v_k$ to a
colored vertex.  By summing over edges of $\gamma$, we obtain
$$\sum_{D \in \D_k} m_{D} = \sum \limits_{e \in \gamma} \sum
\limits_{E: e \in E} m_{D_E} = 0.$$ 
Thus, if $\sum \limits_{D} n_D D$ is a Cartier divisor, then $D$ is a
combination of $D_1,\ldots, D_g$ and so $\sum_{D} m_Dn_D = 0$.  On the
other hand, the set of $(m_D)_D \in \mathbb{Z}^w$ that satisfies the
condition in the Theorem forms a sublattice with dimension at least $w
- g$, since the conditions are linearly independent.  Therefore, the
space of $n_D$ satisfying the condition in the Theorem form a lattice
of dimension $g$, which is the same as that of the space of Cartier
boundary divisors.  This shows that the two spaces are the same, up to
torsion.

To show that the lattices are in fact the same, suppose that $(n_D)_D
\in \mathbb{Z}^w$ satisfies the condition in the Theorem.  By the
previous paragraph, $$\sum \limits_{D} n_DD = \frac{1}{s}\sum
\limits_{i = 1}^{g} r_i D_i$$ for some integers $r_i$ and $s$ such
that $s > 0$.  It is suffices now to show that $s | r_i$ for every $1
\leq i \leq g $.  To see this, note that 
$$\sum \limits_{D} n_D D = \frac{1}{s} \left(\sum \limits_{i = 1}^g
(r_i \sum \limits_{D \in \D_i} D) \right) = \frac{1}{s} \left(\sum
\limits_{D}(\sum \limits_{i: D \in \D_i} r_i)D \right).$$ Thus
$$n_D = \frac{1}{s}\sum \limits_{i: D \in
  \D_i} r_i.$$
For the principal vertex $v_g$, define $D^0 = D_E$ where $E = \lbrace
d_1, ..., d_p \rbrace$.  More generally, for any vertex $v_k$, let
$\gamma_k$ be the down-path from $v_g$ to $v_k$, and let $D^k$ to be
the divisor $D^k = D_E$ where $E$ is the set of edges immediately
below the vertices in $\gamma_k$.  Then
\begin{equation}
n_{D^k} = \frac{1}{s} \sum \limits_{v_i \in \gamma_k} r_i \in \mathbb{Z}.
\end{equation} 
Since $n_{D^1} \in \mathbb{Z}$, we obtain $s | r_1$, and similarly for
any vertices adjacent to the semi-infinite edges besides the root
edge.  Induction on the length of the path
$\gamma_k$ gives $s | r_k$.
\end{proof}

\begin{corollary} \label{charcor}   A sum $\sum \limits_{\lbrace I_1, ..., I_r \rbrace \in \Par(\Gamma)}
  n_{I_1, ..., I_r} D_{I_1, ..., I_r}$ is a Cartier divisor of
  $V(\Gamma)$ if and only if $(n_{ I_1, ..., I_r})_{\lbrace I_1, ...,
    I_r \rbrace \in \Par(\Gamma)}$ is in the orthogonal complement of
  the kernel of $t_\Gamma$, where
$$ t_\Gamma: \Z^{\Par(\Gamma)} \to \Z^{\PP(\Gamma)}, \quad
(t_\Gamma(m))(S)= \sum_{ S \in \{I_1,\ldots, I_r \}}
m_{I_1,\ldots,I_r} .$$
\end{corollary} 

\section{Global description of boundary divisors}

In this section we give a criterion for a boundary divisor in the
moduli space of scaled lines $\ol{M}_{n,1}(\bA)$ to be Cartier.  By
the local description of the moduli space in Section \ref{local}, any
divisor of type I is Cartier, so it suffices to consider divisors of
type II.  To describe the answer, let $I = \{1,\ldots,n \}$, $\Par(I)$
the set of {\em non-trivial} partitions of $I$, and $\PP(I)$ the power
set of {\em non-empty} subsets of $I$.  We identify the set of prime
Weil boundary divisors of type I with the subset of elements of
$\PP(I)$ of size at least two, and the prime Weil boundary divisors of
type II with $\Par(I)$.  Thus in particular the space of Weil boundary
divisors of type II becomes identified with $\Z^{\Par(I)}$, by the map
$$ \left\{ \sum_P {l}(P) D_P \right\} \to \Z^{\Par(I)}, \quad \sum_P {l}(P) D_P
\mapsto {l} .$$
Let $Z(I)$ denote the natural incidence relation,
$$Z(I) = \{ (S,P) \in \PP(I) \times \Par(I) \, | \, S \in P \} .$$
We have a natural map from the space of functions on
$\Par(I)$ to functions on $\PP(I)$ given by pullback and push-forward:
\begin{equation}\label{eq:t} t : \Z^{\Par(I)} \to \Z^{\PP(I)}, \quad (t( h))(S) = \sum_{S\in P} h(P) .\end{equation}
A {\em relation} on the group of Cartier boundary divisors is a
collection of coefficients $ \{ m_{ I_1,\ldots, I_r } \} \in
\Z^{\Par(\{1,\ldots, n\})}$ such that
$$ \sum_{I_1,\ldots, I_r}  m_{I_1,\ldots,I_r} {l}_{I_1,\ldots, I_r}  = 0 $$
for every Cartier divisor $D = \sum_{I_1,\ldots,I_r}
{l}_{I_1,\ldots,I_r} D_{I_1,\ldots, I_r}$.  The space of relations
forms a subgroup of $\Z^{\Par(I)}$.

\begin{theorem} \label{main}  {\rm (Relations on Cartier boundary divisors)}  
The group of relations on the group of Cartier boundary divisors of
type II is the kernel of $t$.
\end{theorem} 

\begin{example}  For $n = 2$ there are two boundary divisors, and there is only the zero relation.  For $n=3$ there are eight boundary divisors, and there is only the zero relation.  For $n=4$ there are 
$| \PP( \{1,2,3,4 \})| - 4 = 11 $
boundary divisors of type I, and 
$| \Par( \{1,2,3,4 \})| = 1 + 6 + 3 + 4 = 14$ 
boundary divisors of type II.  A divisor 
$$D = \sum {l}_{I_1,\ldots,I_r} D_{I_1,\ldots, I_r}$$ 
is Cartier only if the three relations (as $i,j,k,l$ vary)
\begin{equation} \label{simplerelation}
 {l}_{ \{ i,j \}, \{ k \}, \{ l \} } + {l}_{ \{ i \}, \{ j \}, \{ k,l \}}
- {l}_{ \{ i,j \}, \{ k,l \}} - {l}_{ \{ i \}, \{ j \}, \{ k \}, \{ l \}} \end{equation} 
hold. Thus the space of Cartier boundary divisors of type II is a
$11$-dimensional subspace of the space of the $14$-dimension space of
Weil boundary divisors of type II.
\end{example}

\begin{definition} {\rm (Compatible subsets and partitions with a tree)} 
\begin{enumerate} 
\item A tree $\Gamma$ is {\em simple} if it has a single vertex.
\item For each partition $\lbrace I_1, ..., I_r \rbrace$ of $I =
  \lbrace 1, ..., n \rbrace$, define the tree $\Gamma_{I_1, ..., I_r}$
  as follows: $\Gamma_{I_1, ..., I_r} $ has $r$ principal subtrees
  which are respectively the simple colored trees $\Gamma_j, j =
  1,\ldots, r$ whose semi-infinite edges labeled by $i \in I_j$.
\item 
For each subset $I \subset \{ 1, \ldots,n \}$, let $\Gamma_I$ denote
the colored tree with a single colored vertex and a single non-colored
vertex with semi-infinite edges labelled by $i \in I$.
\item Given $v, \tilde{v} \in \Ver(\Gamma)$, we write $v E \tilde{v}$
  if there is an edge connecting $v$ and $\tilde{v}$. A {\em tree
    homomorphism} $f: \Ver(\Gamma) \rightarrow \Ver(\Gamma')$ is a map
  that maps the vertices and edges of $\Gamma$ to the vertices and
  edges of $\Gamma'$ respectively and satisfies:
\begin{enumerate}
\item $f$ maps the principal vertex $v_g$ of $\Gamma$ to the principal
  vertex $v'_0$ of $\Gamma'$.
\item If $v, \tilde{v} \in \Ver(\Gamma)$ satisfies $v E \tilde{v}$, then either $f(v) E f(\tilde{v})$ or $f(v) = f(\tilde{v})$.
\item $f$ maps the colored vertices of $\Gamma$ to the colored
  vertices of $\Gamma'$.
\end{enumerate}
\item  A subset $I \subset \{ 1,\ldots n \}$ is {\em compatible} with $\Gamma$ 
if there exists a tree homomorphism $\Gamma \to \Gamma_I$.
\item 
A partition $\lbrace I_1, ..., I_r \rbrace$ of $\lbrace 1, ..., n
\rbrace$ is {\em compatible} with $\Gamma$ if there exists a tree
homomorphism $f: \Gamma \rightarrow \Gamma_{I_1, ..., I_r}$. 
\end{enumerate} 
\end{definition}

See Figure \ref{sometrees} for the trees $\Gamma_{I_1,\ldots,I_r}$ and
$\Gamma_I$.  Denote by $\Par(\Gamma)$ the set of compatible partitions
of $\{ 1 ,\ldots, n \}$, and by $\PP(\Gamma)$ the set of compatible
subsets of $\{ 1, \ldots, n \}$.

\begin{figure}
\begin{picture}(0,0)%
\includegraphics{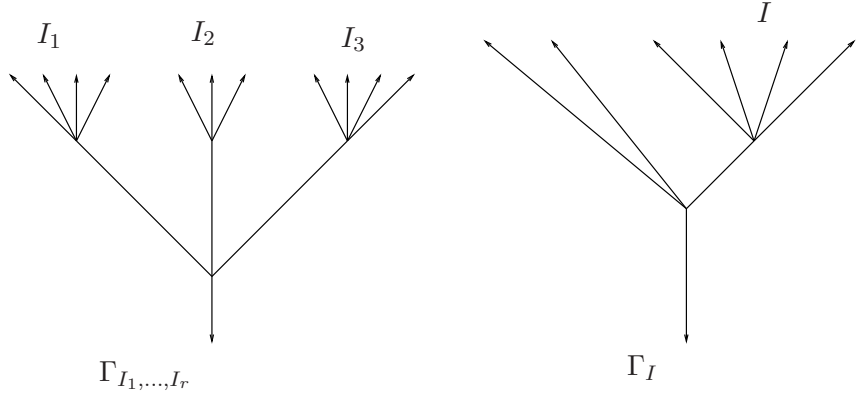}%
\end{picture}%
\setlength{\unitlength}{3947sp}%
\begingroup\makeatletter\ifx\SetFigFont\undefined%
\gdef\SetFigFont#1#2#3#4#5{%
  \reset@font\fontsize{#1}{#2pt}%
  \fontfamily{#3}\fontseries{#4}\fontshape{#5}%
  \selectfont}%
\fi\endgroup%
\begin{picture}(5345,2481)(-11,-2240)
\put(177,-35){\makebox(0,0)[lb]{{{{$I_1$}%
}}}}
\put(1145,-19){\makebox(0,0)[lb]{{{{$I_2$}%
}}}}
\put(2087,-57){\makebox(0,0)[lb]{{{{$I_3$}%
}}}}
\put(4700,130){\makebox(0,0)[lb]{{{{$I$}%
}}}}
\put(3884,-2122){\makebox(0,0)[lb]{{{{$\Gamma_I$ }%
}}}}
\put(569,-2186){\makebox(0,0)[lb]{{{{$\Gamma_{I_1,\ldots,I_r}$}%
}}}}
\end{picture}%
\caption{The trees $\Gamma_I$ and $\Gamma_{I_1,\ldots,I_r}$}
\label{sometrees}
\end{figure}

\begin{proposition} \label{bij2}
There is a canonical bijection between the set of minimally complete
subsets of $\Edge_{< \infty}(\Gamma)$ and the set of compatible
partitions $\Par(\Gamma)$.
\end{proposition}

\begin{example} For the tree $\Gamma$ in Figure \ref{fourex}, the correspondence between 
the minimally complete subsets of $\Edge_{< \infty}(\Gamma)$ and the compatible 
partitions of $\{1, ..., n \}$ is 
$$\lbrace e_1,e_2 \rbrace \longleftrightarrow \lbrace 1, 2 \rbrace, \lbrace 3, 4 \rbrace, \quad \lbrace e_1,e_5,e_6 \rbrace \longleftrightarrow \lbrace 1, 2 \rbrace, \lbrace 3 \rbrace, \lbrace 4 \rbrace$$
$$\lbrace e_2,e_3,e_4 \rbrace \longleftrightarrow \lbrace 1 \rbrace, \lbrace 2 \rbrace, \lbrace 3, 4 \rbrace, \quad \lbrace e_3, e_4, e_5, e_6 \rbrace \longleftrightarrow \lbrace 1 \rbrace, \lbrace 2 \rbrace, \lbrace 3 \rbrace, \lbrace 4 \rbrace.$$ 
\end{example}

\begin{proof}
Given a minimally complete subset $E$, we obtain a partition by
removing the edges in $E$ and considering the partition of the
semi-infinite edges induced by the decomposition into connected
components; that is, two semi-infinite edges are in the same set in
the partition if they can be connected by a path in the complement of
$E$.
There is a morphism
of trees $\Gamma \to \Gamma_{I_1,\ldots, I_r}$ given by collapsing
each connected component of $\Gamma - E$ to a point, which shows that
the partition is compatible.  Conversely, given a compatible
partition, consider the corresponding morphism of trees $\Gamma \to
\Gamma_{I_1,\ldots,I_r}$ and let $E$ denote the subset of edges of
$\Gamma$ that are not collapsed under the morphism.  Since the finite
edges of $\Gamma_{I_1,\ldots, I_r}$ form a minimally complete subset
of $\Edge_{< \infty}(\Gamma_{I_1,\ldots,I_r})$, the set $E$ is also
minimally complete.  The reader may check that these two maps of sets
are inverses.
\end{proof} 

From Proposition \ref{bij2} we obtain a bijective correspondence
between compatible partitions and the prime invariant Weil divisors of
$V(\Gamma)$. For each compatible partition $\{ I_1, ..., I_r \} \in
\Par(\Gamma)$, denote by $D_{I_1, ..., I_r}$ the corresponding
invariant prime Weil divisor of $V(\Gamma)$.

\begin{lemma}  {\rm (Four-term relation)} 
 \label{atleast} Suppose that $ \{ I_1,\ldots, I_r \}$ is 
a partition with at least two elements of size at least two.  Then there exists
a colored tree $\Gamma$ so that $\{ I_1,\ldots, I_r \} \in
\Par(\Gamma)$, and a relation $m \in \ker t_\Gamma$ so that
\begin{enumerate}
\item $m( \{ I_1,\ldots, I_r \} ) = 1$. 
\item For any partition $\{ J_1,\ldots, J_{r'}) \in \Par(\Gamma)$
  distinct from $\{ I_1,\ldots, I_r \}$, we have $m( \{ J_1,\ldots,
  J_{r'} \}) = 0$ unless $r' > r$.
\end{enumerate} 
\end{lemma} 

\begin{proof}   Without loss of generality suppose that $|I_1|, |I_2|$ are both 
at least two, and so admit partitions $I_1 = I_1' \cup I_1''$, $I_2 =
I_2' \cup I_2''$.  Let $\Gamma$ be the tree with $r + 2$ colored
vertices, as in Figure \ref{fourpar}.  Then the
sum of delta functions
$$ \delta_{ I_1',I_1'',I_2',I_2'',I_3,\ldots, I_r } - \delta_{
  I_1,I_2',I_2'',I_3,\ldots, I_r } - \delta_{
  I_1',I_1'',I_2,I_3,\ldots, I_r } + \delta_{ I_1,I_2,I_3,\ldots,
  I_r } \in \ker t_\Gamma $$
is a relation since each subset in each partition occurs an equal
number of times with opposite signs.
\end{proof} 

\begin{figure}
\begin{picture}(0,0)%
\includegraphics{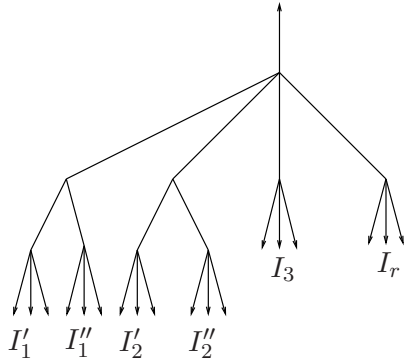}%
\end{picture}%
\setlength{\unitlength}{3947sp}%
\begingroup\makeatletter\ifx\SetFigFont\undefined%
\gdef\SetFigFont#1#2#3#4#5{%
  \reset@font\fontsize{#1}{#2pt}%
  \fontfamily{#3}\fontseries{#4}\fontshape{#5}%
  \selectfont}%
\fi\endgroup%
\begin{picture}(2512,2242)(760,-2591)
\put(2424,-2090){{{{$I_3$}%
}}}
\put(3093,-2062){{{{$I_r$}%
}}}
\put(1909,-2554){{{{$I_2''$}%
}}}
\put(1473,-2554){{{{$I_2'$}%
}}}
\put(1132,-2546){{{{$I_1''$}%
}}}
\put(775,-2552){{{{$I_1'$}%
}}}
\end{picture}%
\caption{The tree $\Gamma$ compatible with the four-term relation} 
\label{fourpar}
\end{figure} 

\begin{lemma} \label{support} Let $m \in \Ker(t)$ be a relation and $r\in\{1,\ldots,n-1\}$. Assume that $m$ vanishes on every partition of length less than $r$. Then $m$ vanishes on every partition that consists of $r-1$ singletons and a set of size $n-r+1$. If $r=n-1$ then $m$ is constantly equal to zero.
\end{lemma} 

\begin{proof} To prove the first assertion, let $S$ be a set of size $n-r+1$. We denote by $P$ the unique partition consisting of $S$ and $r-1$ singletons. Every partition containing $S$, other than $P$, has length less than $r$. Hence by assumption, $m$ vanishes on such a partition. Using the hypothesis $tm=0$ and the definition (\ref{eq:t}) of $t$, it follows that $m(P)=0$. The first assertion follows.

To prove the second assertion, consider the case $r=n-1$. Every partition of length $n-1$ consists of $n-2$ singletons and one set of size two. By the first assertion, $m$ vanishes on every such partition. Since by hypothesis $m$ vanishes on every partition of length less than $n-1$, it follows that $m$ vanishes on all partitions, except possibly on $\big\{\{1\},\ldots,\{n\}\big\}$. However, since $tm=0$, equality (\ref{eq:t}) with $S=\{1\}$ implies that $m$ vanishes on this partition, as well. This proves the second assertion.
\end{proof} 

\begin{proof}[Proof of Theorem \ref{main}]
A Weil divisor is Cartier if and only if its restriction to every
affine Zariski open subset is Cartier.  Hence it suffices to check if
a divisor
$$D = \sum
{l}_{I_1,\ldots, I_r} D_{I_1,\ldots,I_r}$$ 
is Cartier in every chart in Proposition \ref{toricprop}.  Note that each
$D_{I_1,\ldots,I_r}$ corresponds to a non-empty Weil boundary divisor
in $V(\Gamma)$ iff $\{ I_1,\ldots, I_r \} \in \Par(\Gamma)$.  A
criterion for a Weil boundary divisor in $V(\Gamma)$ to be Cartier is
given above in Lemma \ref{charcor}. There is a natural embedding
$\pi_\Gamma$ of $\Ker t_\Gamma$, as in \ref{charcor}, in $\Ker t$
which preserves $m_{I_1, ..., I_r}$ if $\lbrace I_1, ..., I_r\rbrace
\in \Par(\Gamma)$ and maps the other $m_{I_1, ..., I_r}$, where
\ccomment{fixed spacing} 
$\lbrace I_1, ..., I_r \rbrace \notin \Par(\Gamma)$, to $0$.  The
image of $\pi_\Gamma$ is a subspace of $\Ker t$ and $D$ is a
Cartier divisor of $V(\Gamma)$ if and only if $({l}_{I_1, ...., I_r})
\in \coker \pi_\Gamma$.  Hence, $D$ is a Cartier divisor of all
$V(\Gamma)$ if and only if $({l}_{I_1, ..., I_r})$ is in the orthogonal
complement of the image $\pi_\Gamma$ in $\Z^{\Par(I)}$ for all
$\Gamma$.  Thus, it suffices to show that
\begin{equation}\label{eq:Ker t}\Ker t \subset \on{hull}_{\Z}
\on{image} \pi_\Gamma.\end{equation}
For this let $m \in \Ker t$ be a relation. Assume that $m$ is nonzero on some partition of length $\leq n-2$, and let $\{ I_1,\ldots, I_r \}$ a partition of minimal length, on which $m$ is nonzero. It follows from Lemma \ref{support} that $\{ I_1,\ldots, I_r \}$ contains at least two sets of size at least two. Hence by Lemma \ref{atleast} there exists a colored tree $\Gamma$ such that $\{ I_1,\ldots, I_r \}\in\Par(\Ga)$, and a relation $m'\in\ker t_\Gamma$ that attains the value 1 on $\{ I_1,\ldots, I_r \}$ and vanishes on all other partitions of length at most $r$. The relation $m-m_{I_1,\ldots, I_r}m'$ is non-zero on fewer partitions of length $r$ than $m$.  Continuing in this way we obtain a relation which vanishes on all partitions of length less than $n-1$. By the second statement in Lemma \ref{support}, any such relation must be zero. It follows that $m$ is a linear combination of elements of $\ker t_\Gamma$, where $\Gamma$ ranges over all colored trees. This proves the inclusion (\ref{eq:Ker t}) and completes the proof of Theorem \ref{main}.
\end{proof}

\def\cprime{$'$} \def\cprime{$'$} \def\cprime{$'$} \def\cprime{$'$}
  \def\cprime{$'$} \def\cprime{$'$}
  \def\polhk#1{\setbox0=\hbox{#1}{\ooalign{\hidewidth
  \lower1.5ex\hbox{`}\hidewidth\crcr\unhbox0}}} \def\cprime{$'$}
  \def\cprime{$'$}

\end{document}